\documentclass[a4paper, 12pt]{amsart}

\usepackage[utf8]{inputenc}
\usepackage{lmodern}
\usepackage{mathrsfs}
\usepackage[english]{babel}
\usepackage{graphicx}
\usepackage{color}
\usepackage{enumerate}
\usepackage{mdwlist}
\usepackage{microtype}

\usepackage[all]{xy}
\usepackage{tikz-cd}
\usepackage{verbatim}
\usepackage{amsmath,amsfonts,amssymb,amsthm}
\usepackage{bm}
\usepackage{hyperref}

\theoremstyle{plain}
\newtheorem{theorem}{Theorem}[section]
\newtheorem{lemma}[theorem]{Lemma}
\newtheorem{proposition}[theorem]{Proposition}
\newtheorem{corollary}[theorem]{Corollary}
\theoremstyle{definition}
\newtheorem{definition}[theorem]{Definition}
\newtheorem{remark}[theorem]{Remark}
\newtheorem{example}[theorem]{Example}

\newcommand{\ZZ}{\mathbb{Z}} 
\newcommand{\QQ}{\mathbb{Q}} 
\newcommand{\RR}{\mathbb{R}} 
\newcommand{\iso}{\cong}     
\newcommand{\PP}{\mathbb{P}} 
\renewcommand{\AA}{\mathbb{A}} 
\newcommand{\kk}{\Bbbk}
\newcommand{\sheaf}[1]{\mathscr{#1}} 
\newcommand{\Gm}{\mathbb{G}_m} 

\newcommand{\tensor}{\otimes}

\newcommand{\cL}{\mathscr{L}}
\newcommand{\cM}{\mathscr{M}}

\newcommand{\SymXC}{\mathrm{Sym}^n(X[n]/C[n])}
\newcommand{\HilbXC}{\mathrm{Hilb}^n(X[n]/C[n])}
\newcommand{\PrXC}{X[n] \times_{C[n]} \cdots \times_{C[n]} X[n]}

\newcommand{\vv}{{\bm{\mathrm v}}}

\newcommand{\lra}{\longrightarrow}

\DeclareMathOperator{\codim}{codim}

\DeclareMathOperator{\Hilb}{Hilb}
\DeclareMathOperator{\src}{src}
\DeclareMathOperator{\tgt}{tgt}

\title[Geometry of degenerations of Hilbert schemes]{The geometry of degenerations of Hilbert schemes of points}

\begin{document}

\author[M.~G.~Gulbrandsen]{Martin G. Gulbrandsen}
\address{University of Stavanger\\
Department of Mathematics and Natural Sciences\\
4036 Stavanger\\
Norway}
\email{martin.gulbrandsen@uis.no}

\author[L.~H.~Halle]{Lars H. Halle}
\address{University of Copenhagen\\ 
Department of Mathematical Sciences\\
Universitetsparken 5\\ 
2100 Copenhagen\\ 
Denmark}
\email{larshhal@math.ku.dk}

\author[K.~Hulek]{Klaus Hulek}
\address{Leibniz Universit\"at Hannover\\
Institut f\"ur algebraische Geometrie\\
Welfengarten 1\\
30167 Hannover\\
Germany}
\email{hulek@math.uni-hannover.de}

\author[Z.~Zhang]{Ziyu Zhang}
\address{Leibniz Universit\"at Hannover\\
Institut f\"ur algebraische Geometrie\\
Welfengarten 1\\
30167 Hannover\\
Germany}
\curraddr{ShanghaiTech University\\
Institute of Mathematical Sciences\\
393 Middle Huaxia Road\\
Shanghai 201210\\
P.R.China}
\email{zhangziyu@shanghaitech.edu.cn}

\keywords{Geometric Invariant Theory, degeneration, Hilbert scheme, dual complex}

\subjclass[2010]{Primary: 14D06; Secondary: 14C05, 14L24, 14D23}

\begin{abstract}
Given a strict simple degeneration $f \colon X\to C$ the first three authors previously constructed a degeneration $I^n_{X/C} \to C$ of the relative degree $n$ Hilbert scheme of $0$-dimensional subschemes.
In this paper we investigate the geometry of this degeneration, in particular when the fibre dimension of $f$ is at most $2$. In this case we show that $I^n_{X/C} \to C$ is a dlt model.
This is even a good minimal dlt model if $f \colon X \to C$ has this property. We compute the dual complex of the central fibre $(I^n_{X/C})_0$  and relate this to the essential skeleton of the generic fibre.  
For a type II degeneration of $K3$ surfaces we show that the stack ${\mathcal I}^n_{X/C} \to C$ carries a nowhere degenerate relative  logarithmic $2$-form. Finally we discuss the relationship
of our degeneration with the constructions of Nagai.   
\end{abstract}

\maketitle


\section{Introduction}\label{sec:intro}

The Hilbert scheme parameterizing zero dimensional subschemes of length $n$ on a variety $V$ is a much studied object, and appears in a wide range of contexts such as enumerative geometry, representation theory and mathematical physics, to name a few. In certain cases, the geometry of $V$ is reflected in the Hilbert scheme; for instance, if $V$ is smooth, irreducible and of dimension at most two, then it is well known that $ \mathrm{Hilb}^n(V) $ is again smooth and irreducible and of dimension $ n \cdot \mathrm{dim}(V)$. On the other hand, if $\mathrm{dim}(V) > 2$ the geometry of the Hilbert scheme can become arbitrarily bad as $n$ grows, and, to our best knowledge, little is known in general if $V$ is singular and $\mathrm{dim}(V) > 1 $.

In previous work \cite{GHH-2019}, the first named three authors studied the question of how $ \mathrm{Hilb}^n(V) $ degenerates along with the underlying variety $V$. As formulated, this question is obviously too broad; in order to obtain a useful answer, one needs to impose strong conditions on the degenerations being used. In \cite{GHH-2019}, so-called \emph{strict simple degenerations} were considered. This roughly means a flat morphism $ f \colon X \to C $ where $ 0 \in C $ is a smooth pointed curve, with smooth connected fibres except for $ X_0 = f^{-1}(0) $ which forms a strict normal crossings divisor without triple intersections. A canonical degeneration is given in this case by the relative Hilbert scheme $ \mathrm{Hilb}^n(X/C) \to C $. However, due to the presence of singularities in $X_0$, it seems hopeless to control the geometry of this family if either $n$ is large or if the relative dimension of $f$ is larger than $1$ (with the exception that $n$ and the fibre dimension are equal to $2$). 

The main construction in \cite{GHH-2019} yields a different degeneration $I^n_{X/C} \to C$, which coincides with the relative Hilbert scheme over $ C \setminus \{0\} $, but which seems to have far better geometric properties in general. A key ingredient in this construction is Li's $G[n] (\cong (\mathbb{G}_m)^n) $-equivariant \emph{expanded degeneration} $ X[n] \to C[n] $ of the strict simple degeneration $X \to C $; in fact, $I^n_{X/C}$ is obtained as a certain GIT quotient of $\mathrm{Hilb}^n(X[n]/C[n])$ by the torus $G[n]$. The crucial advantage of this approach is that GIT (semi-)stable subschemes can only be supported on the \emph{smooth} locus of the fibres of Li's expansion and that (semi-)stability can be determined by a simple combinatorial consideration. We moreover proved that the stable and semi-stable loci coincide. 
 
It should be remarked that the results in \cite{GHH-2019} are quite general, and make no assumption on the fibre dimension of $ X \to C $. The purpose of this paper is to study the degeneration $ I^n_{X/C} $ in detail in the case where the fibres $ X_c $ have dimension at most $2$, in which case the Hilbert scheme $ \mathrm{Hilb}^n(X_c) $ of a general fibre is a smooth variety. We shall focus especially on the (birational) geometry of $ I^n_{X/C} $, and on the geometry and the combinatorial structure of the degenerate fibre $ (I^n_{X/C})_0 $.

\subsection{The main results}\label{sec:intro-mainresults}

Our first main result concerns the birational geometry of the degeneration $ I^n_{X/C} \to C $. Even though $ X \to C $ is a semi-stable degeneration, it is too much to hope for in general that semi-stability passes on to the Hilbert scheme level via our construction. However, $ I^n_{X/C} \to C $ still turns out to have mild, controllable singularities from the viewpoint of the Minimal Model Program.

To explain this, we need to recall some terminology for pairs $(Y,D)$ consisting of a normal $\mathbb{Q}$-factorial variety $Y$ and (for simplicity) a reduced divisor $ D $. Then $(Y,D)$ is called log canonical (lc) if the discrepancy $ a(E,Y,D) \geqslant -1$ for every exceptional divisor $E$ over $Y$. An lc pair $(Y,D)$ is moreover called divisorial log terminal (dlt) if equality occurs if and only if $\mathrm{centre}_Y(E) $ intersects nontrivially the (open) locus where $(Y,D)$ is snc. If $Y$ admits a morphism $ f \colon Y \to C $ to a smooth pointed curve $ 0 \in C $, and $ D = f^{-1}(0)_{\mathrm{red}} $ (with otherwise smooth fibres), one calls $Y \to C $ a dlt model if $ (Y,D)$ is dlt.

\begin{theorem}
Assume that the relative dimension of $ X \to C $ is at most $2$. Then $ I^n_{X/C} $ is normal and $\mathbb{Q}$-factorial, the special fibre $ (I^n_{X/C})_0 $ is a reduced divisor, and $ I^n_{X/C} \to C $ is a dlt model. 
\end{theorem}

In order to establish this result (Theorem \ref{thm:dltmodel} in the text), we first observe that the GIT semi-stable locus $ \mathrm{Hilb}^n(X[n]/C[n])^{ss} $ forms a semi-stable degeneration when viewed as a scheme over the base curve $C$, thus $ \mathrm{Hilb}^n(X[n]/C[n])^{ss} \to C $ is, in particular, a dlt model. We then perform a careful analysis of the $G[n]$-action to ensure that the dlt property persists when passing to the GIT quotient $I^n_{X/C}$ (this is far from obvious, due to the presence of (finite) non-trivial stabilizer groups).

A natural, closely related, question we consider is when $ I^n_{X/C} \to C $ is minimal; a projective dlt model $Y \to C$ is called a good minimal model if $Y$ is $\mathbb{Q}$-factorial and if the divisor $ K_Y + (Y_0)_{\mathrm{red}} $ is semi-ample over $C$. We show in Corollary \ref{cor:goodminimal} that our Hilbert scheme degeneration has the desirable property of preserving minimality: 

\begin{corollary}
Assume that the strict simple degeneration $ X \to C $ is a good minimal dlt model. Then also $ I^n_{X/C} \to C $ is a good minimal dlt model. 
\end{corollary}

The combinatorial structure of the special fibre $ (I^n_{X/C})_0 $ is recorded in its \emph{dual complex} $\mathcal{D}((I^n_{X/C})_0)$. Generalizing the well known construction for an snc divisor, the dual complex $\mathcal{D}(E)$ of a dlt divisor $E = \sum_{i \in I} E_i$ is built up from gluing $d$-dimensional simplices corresponding to the connected components of the various $(d+1)$-fold intersections of the $E_i$-s, where $ 1 \leqslant d < \vert I \vert$ (see Subsection \ref{subsec:dualcomplexdef} for a precise definition). By our assumption that $ X \to C $ is a strict simple degeneration, its dual complex is a graph $\Gamma$ (which we refer to as the \emph{dual graph}). Our next main result, Theorem \ref{thm:dualcomplex-hilb}, yields the following description of the dual complex attached to $ I^n_{X/C} \to C $:  

\begin{theorem}
The dual complex $\mathcal{D}((I^n_{X/C})_0) $ is isomorphic, as a $\Delta$-complex, to the $n$-th symmetric product $ \mathrm{Sym}^n(\Gamma)$ of $\Gamma$. 
\end{theorem}

Our proof of this result is conceptual, and it might be useful to briefly explain the strategy here. First of all, we observe that the GIT analysis and results in \cite{GHH-2019} can be easily extended to the $n$-fold product $ X[n] \times_{C[n]} \ldots \times_{C[n]} X[n] $ and the symmetric product $\mathrm{Sym}^n(X[n]/C[n]) $, giving GIT quotients, denoted $P^n_{X/C}$ and $J^n_{X/C}$, respectively, which are related by a $\mathfrak{S}_n $-quotient map $P^n_{X/C} \to J^n_{X/C}$. By a direct computation, we show that $\mathcal{D}((P^n_{X/C})_0) \cong (\Gamma)^n $, and that the formation of the dual complex commutes with the $\mathfrak{S}_n $-action so that $\mathcal{D}((J^n_{X/C})_0) \cong \mathrm{Sym}^n(\Gamma) $. On the other hand, we also show that the Hilbert-Chow morphism on the smooth fibres extends to a map $ I^n_{X/C} \to J^n_{X/C} $, which, in particular, induces an isomorphism of dual complexes, from which the theorem follows.

\vspace{0.2 cm}

Note that most of the results explained so far do not make strong assumptions on $X \to C$ apart from the requirement that the fibre dimension is at most $2$. However, in the case where $ X \to C $ is, say, a (projective) type II degeneration of K3 surfaces, it is natural to wonder how the symplectic structure of $ \mathrm{Hilb}^n(X_c) $, with $X_c$ a general smooth fibre, degenerates as $c$ tends to $ 0 $ in $ C $. In order to address this question, we work with the stack quotient $\mathcal{I}^n_{X/C}$ rather than the GIT quotient $I^n_{X/C}$. The reason for this is that $ \mathcal{I}^n_{X/C} \to C $ is a semi-stable degeneration of DM-stacks, hence comes equipped with sheaves of relative logarithmic $d$-forms which are locally free for all $d \geqslant 0$. In this setting, we can prove that the symplectic $2$-form in the generic fibre extends to a log $2$-form in the family. We remark that the presence of quotient singularities prevents $ I^{n}_{X/C} \to C $ to be semistable, or, more generally, log smooth.

\begin{proposition}
The stack $\mathcal{I}^n_{X/C}$ is proper and semi-stable over $C$, and carries, if $K_{X/C}$ is trivial, an everywhere non-degenerate relative logarithmic $2$-form.
\end{proposition}

\subsection{Related work}
We would like to take the opportunity in this paragraph to comment on related results, by several groups of authors, that have appeared during the writing up of this paper.

First of all, Nagai presents in his recent papers \cite{nagai-2018} and \cite{nagai-2018} an entirely  different approach to the construction of our Hilbert scheme degeneration $ I^n_{X/C} \to C $, with $ X \to C $ a strict simple degeneration as usual. Using toric methods, he describes the local structure of the singularities in $ \mathrm{Sym}^n(X/C) $, building on the analysis (in \cite{nagai-2018}) in the case of the local model $ \mathbb{A}^3 \to \mathbb{A}^1; (x,y,z) \mapsto xy $. Based on this, Nagai goes on to produce an explicit $\mathbb{Q}$-factorial terminalization $ Y^{(n)} \to \mathrm{Sym}^n(X/C) $, and, in his main result \cite[Theorem 4.3.1]{nagai-2018}, exhibits an isomorphism $ I^n_{X/C} \to Y^{(n)} $ over $C$. 

In \cite{KLSV-2018}, Koll\'ar \emph{et.~al.} study (among other things) minimal dlt models of $2n$-dimensional Hyperk\"ahler manifolds. Most relevant to our results, is that they establish a strong result linking properties of the dual complex $\mathcal{D}(\mathcal{Z}_0)$ of the special fibre $ \mathcal{Z}_0 $ of a minimal dlt model $ \mathcal{Z}/\Delta $ and the monodromy action on the second cohomology $H^2$ of a general fibre of $ \mathcal{Z} $. More precisely, in \cite[Theorem 0.10]{KLSV-2018} they prove that the dimension of the geometric realization of $\mathcal{D}(\mathcal{Z}_0)$ is $0$, $n$ and $2n$, if and only if the index of nilpotency of the log monodromy operator on $H^2$ is $1$, $2$ or $3$ respectively. Starting with a type II degeneration $ X \to C $ of K3 surfaces, our degeneration $ I^n_{X/C} \to C $ yields a dual complex of dimension $n$, and it is indeed well known that the nilpotency index is $2$ in this case.

If $ Y \to C $ is a minimal dlt model, the dual complex $\mathcal{D}(Y_0)$ can also be identified with the so-called \emph{essential skeleton} associated with the Berkovich analytification of the generic fibre $Y_{\eta}$. The essential skeleton is, in fact, intrinsic to $Y_{\eta}$, and can be studied without reference to an explicit minimal dlt model. In \cite{BM17}, Brown and Mazzon study the formation of the essential skeleton under products and group quotients of varieties, by means of tools from Berkovich geometry and log geometry. Most relevant to this paper is their result \cite[Corollary 6.2.3]{BM17}, which states that for a K3 surface $ X $ over a non-Archimedean field $K$, and with semi-stable reduction over the ring of integers $\mathcal{O}_K$, the essential skeleton of $\mathrm{Hilb}^n(X) $ is PL homeomorphic to the $n$-th symmetric product of the essential skeleton of $X$. (By \cite[Corollary 6.1.4]{HN-2018}, the latter object is PL isomorphic to either a point, a closed interval, or to the standard $2$-sphere.) The reader might want to compare their result with Theorem \ref{thm:Hilb-skeleton-main} in this paper, which provides an alternative proof based on our construction of explicit dlt models in Section \ref{sec:dlt}, in the case of a type II degeneration of K3 surfaces.

\subsection{Organization of the paper}
We end this introduction with a short overview of the paper. First of all, in order to make the paper self-contained, we provide in Section \ref{sec:GIT} a brief overview of the GIT construction in \cite{GHH-2019}. In particular, we recall Li's expansions $X[n] \to C[n] $ of a strict simple degeneration $ X \to C $, as well as the description of the GIT (semi-)stable locus in $ \mathrm{Hilb}^n(X[n]/C[n]) $. In Section \ref{sec:symnew}, we extend the GIT analysis in \cite{GHH-2019} to the $n$-fold product, resp.~the $n$-th symmetric product, of $ X[n] \to C[n]$, and obtain GIT quotients $P^n_{X/C}$, resp.~$J^n_{X/C}$.  These objects are related by a Hilbert-Chow type morphism $I^n_{X/C} \to J^n_{X/C}$ and a $\mathfrak{S}_n $-quotient map $P^n_{X/C} \to J^n_{X/C}$, respectively. Section \ref{sec:strata} is devoted to a careful investigation of the degenerate fibres of the GIT quotients $I^n_{X/C}$, $J^n_{X/C}$ and $P^n_{X/C}$. We construct a stratification of each of these fibres, and we describe how their irreducible components intersect. While being somewhat tedious to derive, these results are essential for all subsequent work in this paper. In Section \ref{sec:dlt} we establish the key result, Theorem \ref{thm:dltmodel}, saying that $I^n_{X/C} \to C$ forms a dlt model. We moreover prove in Corollary \ref{cor:goodminimal} that minimality of $ X \to C $ passes on to $I^n_{X/C} \to C$. In Section \ref{sec:dual-computation}, we prove that the three degenerations $I^n_{X/C}$, $J^n_{X/C}$ and $P^n_{X/C}$ all carry dual complexes. We show by explicit computation in Proposition \ref{prop:dualp} that $\mathcal{D}(P^n_{X/C}) \cong (\Gamma)^n $ (with $\Gamma$ the dual graph of $ X \to C $) and derive from this in Proposition \ref{prop:duals} that $\mathcal{D}(J^n_{X/C}) \cong \mathrm{Sym}^n(\Gamma)$. Lastly, we show in Theorem \ref{thm:dualcomplex-hilb} that the Hilbert-Chow morphism induces an isomorphism of dual complexes. We moreover give some applications to the Berkovich essential skeleton of the $n$-th Hilbert scheme attached to the generic fibre of a type II degeneration of K3 surfaces. We address in Section \ref{sec:symplectic} the question of how the symplectic structure degenerates along with our Hilbert scheme degenerations. Working on the stack quotient (rather than the GIT quotient) we exhibit a nowhere degenerate relative logarithmic $2$-form. Finally, in Section \ref{sec:nagai} we compare, for $n=2$, our construction with that in \cite{nagai-2008}.

\subsection{Convention}

Throughout this paper, we work over an algebraically closed field $\kk$ of characteristic $0$. Unless otherwise specified, a \emph{point} of a $\kk$-scheme of finite type always means a closed point. In Section \ref{sec:GIT} (except Proposition \ref{prop:GIT-open}) and Section \ref{sec:symnew}, we allow the degeneration family $X \to C$ to have arbitrary relative dimension; from Section \ref{sec:strata} until Section \ref{sec:dual-computation}, we assume that the relative dimension is at most $2$; in Section \ref{sec:symplectic} and Section \ref{sec:nagai}, we restrict further to the case of relative dimension equal to $2$.

\subsection{Acknowledgements}
LHH would like to thank E.~Mazzon and J.~Nicaise for useful discussions
and their interest in this project. MGG thanks the Research Council of
Norway for partial support under grant 230986. KH is grateful to DFG for
partial support under grant Hu 337/7-1. Finally we thank the referee for
their careful reading and useful suggestions.


\section{Review of the GIT construction}\label{sec:GIT}

We start with the basic objects which we will study in this paper.

\begin{definition}\label{def:simple}
	A \emph{strict simple degeneration} over a smooth curve $C$ is a flat morphism $f \colon X\to C$
	from a smooth algebraic space $X$ to $C$, such that
	\begin{enumerate}
		\item[\rm{(i)}] $f$ is smooth outside the central fibre $X_0 = f^{-1}(0)$,
		\item[\rm{(ii)}]  the central fibre $X_0$ has normal crossing singularities and its
		      singular locus $D\subset X_0$ is smooth,
		\item[\rm{(iii)}]  all components of $X_0$ are smooth.
	\end{enumerate}
\end{definition}
     
The last condition in this definition is equivalent to the assumption that
there are no  self-intersections of components of $X_0$. 
Let $\Gamma(X_0)$ be the dual graph  of the central fibre: the vertices of this graph are given by the components of $X_0$ and 
the edges correspond to the irreducible components of the singular locus of $X_0$.
The assumption that $X_0$ has no self-intersections is then equivalent to saying that the graph $\Gamma(X_0)$ has no loops.
Our assumptions imply that $X_0$ has no triple intersections. In terms of degenerations of $K3$ surfaces, one of the main motivations of our paper, this
is saying that we consider type II degenerations, but not type III degenerations.     

For what we want to do, we will further need to choose an {\em orientation}  on the dual graph $\Gamma(X_0)$.
Given such an orientation, we can associate to  any degeneration $ f \colon X \to C$ and any non-negative integer $n$  the {\em expanded degenerations} $ f[n] \colon X[n]\to C[n]$, which were introduced by Li \cite{li-2001} and
extensively studied by Li and Wu \cite{LW-2011}.  
For a discussion of this construction in the context we are concerned with, we refer the reader to \cite[Section 1]{GHH-2019}.  
The assumption that $\Gamma(X_0)$ has no loops implies that $X[n]$ is a scheme provided $X$ is a scheme \cite[Proposition 1.9]{GHH-2019}.
Moreover, if we start with a projective degeneration $X \to C$, then $X[n]\to C[n]$ is projective if and only if the dual graph $\Gamma(X_0)$ has no directed cycles
\cite[Proposition 1.10]{GHH-2019}. 
We also recall that $X[n]\to C[n]$ admits a natural action
of  the $n$-dimensional split torus $G[n]$.

In what follows, we will always assume that the morphism $f:X \to C$ is projective. We must also make the additional assumption that the associated 
dual graph $\Gamma(X_0)$ is {\em bipartite}, i.e., the set of vertices admits a partition into two disjoint subsets such that none of these subsets contains two adjacent vertices. 
A bipartite partition is equivalent to the choice of an
orientation on $\Gamma(X_0)$ such that no vertex is at the same time a source and a target, and exists if and only if $\Gamma(X_0)$ contains 
no cycles of odd length. In this case there are two possible orientations which define a bipartite partition and the two choices differ by reverting all arrows. 

We will from now on assume that $\Gamma(X_0)$ is a bipartite graph and that we have chosen one of the two possible bipartite orientations. 
Although our restriction to degenerations with bipartite graphs imposes a condition on the degenerations which we consider, this is not essential; by \cite[Remark 1.17]{GHH-2019}
one can always perform a quadratic base change on $X \to C$ in order to fulfill this condition. Also, it is irrelevant which of the  two possible orientations one chooses, as they produce isomorphic expanded degenerations by \cite[Proposition 1.11]{GHH-2019}.   
   
The crucial point of \cite{GHH-2019} is the construction of a relatively ample line bundle $\sheaf{L}$ on $X[n] \to C[n]$ together with a $G[n]$-linearization.
The next step is to consider the relative Hilbert scheme $\Hilb^n(X[n]/C[n])$. 
We denote by $\mathcal{Z} \subset
\Hilb^n(X[n]/C[n]) \times_{C[n]} X[n]$ the universal family and by $p$ and
$q$ the first and second projection from $\mathcal Z$ onto the first and second factor  respectively. Then the line bundle
\begin{equation*}
 \sheaf{M}_{\ell} :=
 \mathrm{det}~p_* \left( q^*\sheaf{L}^{\tensor \ell} \right) 
\end{equation*}
is relatively ample when $\ell \gg 0$.  We choose one such $\ell$ and remark that the final construction will not depend on this choice.
The line bundle $ \sheaf{M}_{\ell}$ inherits a $G[n]$-linearization from $\sheaf{L}$.
The degenerations constructed in \cite{GHH-2019} are obtained by a GIT construction. 
For this let $\Hilb^n(X[n]/C[n])^s$ and $\Hilb^n(X[n]/C[n])^{ss}$ be the sets of stable and semi-stable points respectively. It was shown in \cite[Theorem 2.10]{GHH-2019}
that these two sets coincide, see Theorem \ref{thm:GIT-main} below. This theorem also shows that they do not depend on the choice of  $\ell \gg 0$.
We then set
\begin{equation*}
	I^n_{X/C} = \Hilb^n(X[n]/C[n])^{ss}/G[n].
\end{equation*} 
By construction, and the fact that $C[n]/G[n]=C$, we have a morphism $I^n_{X/C} \to C$. We denote by  $X^* \to C^*$  the family away from the origin and similarly for 
$(I^n_{X/C})^*\to C^*$. Then $\Hilb^n(X^*/C^*) \cong (I^n_{X/C})^*$ and in this way $I^n_{X/C} \to C$ can be viewed as a degeneration of the 
Hilbert schemes of the smooth fibres of $X\to C$.
 
A crucial point of this construction is that one has a good understanding of the stability condition: one can formulate an explicit criterion, see \cite[Theorem 2.10]{GHH-2019},
which allows one to determine the (semi-)stable locus explicitly.  In order to recapitulate this we first have to recall the geometry of the 
map $X[n] \to C[n]$. Starting with a local  \'etale coordinate $t$ on $C$ one obtains coordinates $t_i, i \in \{1, \ldots, n+1 \}$ on $C[n]$.
For notational convenience we shall from now on write $[n+1]= \{1, \ldots, n+1 \}$.  
For any subset $I=\{i_1, \ldots, i_r\} \subset[n+1]$ we denote by $C[n]_I$ the locus in $C[n]$ where the coordinates $t_i, i\in I$ vanish and by $X[n]_I$ the preimage 
of $C[n]_I$ in $X[n]$. The structure of the fibre over a point $q\in C[n]$ is  determined by the number of 
coordinate functions which vanish on $q$. To describe this, we recall from \cite[Section 1]{GHH-2019} the following construction. Given the graph $\Gamma=\Gamma(X_0)$ and any subset $I \subset [n+1]$, we 
construct a new graph $\Gamma_I$ as follows:
if $I = \varnothing$, then $\Gamma_I$ contains a single vertex without any arrow; otherwise, we replace each arrow
\begin{equation*}
	\stackrel{v}{\bullet} \xrightarrow{\gamma} \stackrel{v'}{\bullet}
\end{equation*}
in $\Gamma$ with $|I|$ arrows labelled by $I$ in ascending order in the
direction of the arrow:
\begin{equation*}
	\stackrel{v_I}{\bullet} \xrightarrow{i_1} \circ \xrightarrow{i_2} \circ \to \cdots
	\xrightarrow{i_r} \stackrel{v_I'}{\bullet}.
\end{equation*}
(So here $r=|I|$ and $i_1<i_2<\cdots < i_r$ are the elements in $I$.) 
Note that we no longer demand  the new graph  $\Gamma_I$ to be bipartite.
We colour the old nodes black and the new ones white. 
The valence
of a white node is $2$, and  that of a  black node is unchanged from $\Gamma$.
We label the black nodes $v_I$, where $v$ is the corresponding node
in $\Gamma$ and label the white nodes $(I,\gamma,i_{\ell})$, $\ell =1, \ldots r-1$. 
Finally, we extend
the notation by letting
\begin{equation*}
	(I,\gamma,0) = v_I, \quad (I,\gamma,\max I) = (I,\gamma,i_r) = v'_I.
\end{equation*}
In this notation \cite[Proposition 1.12]{GHH-2019} can be rephrased as follows:
\begin{proposition}
The following holds:
\begin{enumerate}
\item[{\rm (i)}] As an algebraic space over $\kk$, $X[n]_I$ is a union of nonsingular components with normal crossings.
\item[{\rm (ii)}] The dual graph of $X[n]_I$ is canonically isomorphic to the graph   $\Gamma_I$.
\end{enumerate}
\end{proposition}

In view of this description we can write $X[n]_I$ as a union of
irreducible components indexed by the vertices of $\Gamma_I$. We
next introduce notation for these components.

Let $Y_v$ denote the component of $X_0$ indexed by the vertex $v$ in
$\Gamma(X_0)$. We denote by $(Y_v)_I$ the component of $X[n]_I$
corresponding to the (black) vertex  $v_I$ in $\Gamma_I$. This is in
fact the component which is mapped birationally onto $Y_v \times_C
C[n]_I$ under the birational map $X[n]\to X\times_C C[n]$. For the
remaining (white) vertices in $\Gamma_I$, let $i\in I$ be the label of
the arrow pointing into that vertex (so $i<\max I$). Then write
$\Delta^{\gamma,i}_I$ for the corresponding component in $X[n]_I$.

For non-empty $I$ it is useful to introduce the synonyms
\begin{equation*}
	\Delta^{\gamma,0}_I = (Y_v)_I, \quad\Delta^{\gamma,\max I}_I= (Y_{v'})_I
\end{equation*}
for the black vertex components. We also define $\Delta^i_I$ as the union of all $\Delta^{\gamma,i}_I$ where $\gamma$ runs through all arrows in $\Gamma$.

We now return to the stability criterion.
Let $[Z] \in \Hilb^n(X[n]/C[n])$ be represented by a subscheme $Z \subset X[n]_q$ for some point
$q \in C[n]$. Then we define the set  
\begin{equation*}
	I_{[Z]} = \{ i \mid t_i(Z) = 0 \} \subset [n+1].
\end{equation*} 
We write $I_{[Z]}=\{a_1, \ldots, a_r\}$ and for notational convenience we also set $a_0=1$ and
$a_{r+1}=n+1$.  Thus we obtain a vector ${\mathbf a}=(a_0, \ldots, a_{r+1}) \in
\ZZ^{r+2}$, which in turn determines a vector $\mathbf{v}_{\mathbf{a}} \in
\ZZ^{r+1}$ whose $i$-th component is $ a_i - a_{i-1} $. The vector $\mathbf{v}_{\mathbf{a}}$ is called the \emph{combinatorial support} of $Z$.

We say that $Z$ has {\em smooth support} if $Z$ is supported on the smooth part of $X[n]_{I_{[Z]} }$.    
When $Z$ has smooth support, then there exists for each $P \in
\mathrm{Supp}(Z)$ a unique integer $0 \leqslant i(P) \leqslant r$ such that $P \in
\Delta^{a_{i(P)}}_{I_{[Z]}}$.
The \emph{numerical support} of $Z$ is then defined as the tuple  
	\begin{equation*}
		\mathbf{v}(Z) = \sum_P n_P  \mathbf{e}_{i(P)} \in \mathbb{Z}^{r+1}, 
	\end{equation*}
where $n_P$ is the multiplicity of the point $P$ in $Z$, and $\mathbf{e}_{i(P)}$ denotes the $i(P)$-th standard basis vector of $
	\mathbb{Z}^{r+1}$.
In this way the numerical support keeps track of the distribution of
the underlying cycle of $Z$ on the strata $\Delta^{a_i}_{I_{[Z]}}$, for $0 \leqslant i \leqslant r$. 

We can now rephrase \cite[Theorem 2.10]{GHH-2019} as follows

\begin{theorem}\label{thm:GIT-main}
A point $[Z] \in \Hilb^n(X[n]/C[n])  $ is stable if and only if it has smooth support and its combinatorial and numerical support coincide: $\mathbf{v}(Z)=\mathbf{v}_{\mathbf{a}}$.
All semi-stable points are stable.
\end{theorem}

Finally, assume that the dimension $d$ of the fibres of $f \colon X \to C $ is at most $2$. We write $ X^{sm} $ for the smooth locus of $ f $. Below we give a comparison of the scheme $I^n_{X/C}$ to $\Hilb^n(X^{sm}/C)$; this result is fundamental to many applications in this paper. Before giving the precise statement, we first recall that both $\Hilb^n(X^{sm}/C) \to C $ and $\Hilb^n(X[n]^{sm}/C[n]) \to C[n]$, where $ X[n]^{sm} $ denotes the smooth locus of $f[n]$, are smooth of relative dimension $ d n $.

\begin{proposition}\label{prop:GIT-open}
Assume that the dimension of the fibres of $ X \to C $ is at most two. Then there is an open inclusion $\Hilb^n(X^{sm}/C) \subset I^n_{X/C}$ whose complement has codimension $2$.
\end{proposition}
\begin{proof}
We denote the pullback of $X^{sm}$ under the map $X[n]\to X$ by $X[n]^{{tr}}$. 
Then  $X[n]^{{tr}}$ is an open subset of $X[n]^{sm}$, the relative Hilbert scheme $\Hilb^n(X[n]^{{tr}}/C[n])$ is $G[n]$-invariant, and the 
stabilizers of all points are trivial. By the stability criterion in Theorem \ref{thm:GIT-main} we moreover obtain an open inclusion 
$$\Hilb^n(X[n]^{{tr}}/C[n]) \subset \Hilb^n(X[n]/C[n])^{ss}. $$
Using \cite[Proposition 0.2]{GIT}, one finds that the quotient
$$\Hilb^n(X[n]^{{tr}}/C[n])/G[n] \subset \Hilb^n(X[n]/C[n])^{ss}/G[n]=I^n_{X/C}$$
is an open subscheme which is naturally isomorphic to $\Hilb^n(X^{sm}/C)$.
The complement $ \mathcal{C} $ of 
$\Hilb^n(X[n]^{{tr}}/C[n])$ in $\Hilb^n(X[n]/C[n])^{ss}$ consists of subschemes of length $n$ in fibres of $ \Hilb^n(X[n]/C[n])^{ss} \to C[n] $ with at least one point in its support belonging to an 
inserted component $ \Delta^{\gamma,i_{\ell}}_I$.  As the inserted components lie over points in $C[n]$ where at least two coordinate functions $t_i$ vanish, it follows that $\mathcal{C}$ has codimension $2$ in $\Hilb^n(X[n]/C[n])^{ss}$.  
By taking a Luna slice we then see that the same holds for 
$$\Hilb^n(X[n]^{{tr}}/C[n])/G[n]  \cong \Hilb^n(X^{sm}/C)$$ 
in $I^n_{X/C}$.
\end{proof}


\section{Comparison to the symmetric product}\label{sec:symnew}

In this section we look at similar GIT constructions on the $n$-fold product and the symmetric product of the family $f[n]: X[n] \to C[n]$, and establish the relation among the semi-stable loci of the relative Hilbert scheme and the above two families.

\subsection{Semi-stable loci of the symmetric product}

We write 
$$f: X \to C$$
and 
$$ f[n]: X[n] \to C[n] $$
for the corresponding expanded degeneration.

As constructed in \cite[Lemma 1.18]{GHH-2019}, we have a $G[n]$-linearized line bundle $\cL$ on $X[n]$, relatively ample over $C[n]$. The symmetric group $\mathfrak{S}_n$ acts on the $n$-fold product $\PrXC$ by permuting its factors, which leads to the following quotient map
\begin{equation}
\label{eqn:tau}
\tau: \PrXC \to \SymXC.
\end{equation}
The following functorial construction gives us a relatively ample line bundle on $\SymXC$.

\begin{lemma}
\label{lem:ampcon}
$\cL$ induces a $G[n]$-linearized line bundle $\cM$ on
$$ \SymXC, $$
relatively ample over $C[n]$.
\end{lemma}

\begin{proof}
	We define the $G[n]$-linearized line bundle
	$$ \cL^{\boxtimes n} = \cL \boxtimes \cL \boxtimes \cdots \boxtimes \cL $$
	on $ X[n] \times_{C[n]} X[n] \times_{C[n]} \cdots \times_{C[n]} X[n] $. Since $\cL$ is relatively ample, $\cL^{\boxtimes n}$ is also relatively ample. The symmetric group $\mathfrak{S}_n$ acts on $\PrXC$ and $\cL^{\boxtimes n}$ by permuting factors. It is clear that the stabilizer of any closed point acts trivially on the fibre of $\cL^{\boxtimes n}$ at that point. It follows from Kempf's descent lemma (see e.g. \cite[Theorem 2.3]{DN89}) that $\cL^{\boxtimes n}$ descends to a line bundle $\cM$ on $\SymXC$; namely, $\tau^*\cM \cong \cL^{\boxtimes n}$ is an $\mathfrak{S}_n$-equivariant isomorphism. 
	Indeed, such an $\cM$ is uniquely determined (see e.g. \cite[Section 3]{Tel00}). To show that $\cM$ is relatively ample, we cover $C[n]$ by affine open subschemes $U_i$'s. The ampleness of $\cM$ on each $\tau^{-1}(U_i)$ follows from \cite[Theorem 1.10(ii)]{GIT}. 
\end{proof}

\begin{remark}
We write any closed point $[Z] \in \SymXC$ as a positive linear combination of closed points in $X[n]$ in the form
$$ [Z]=\sum_P n_P[P].$$
Then the fibre of $\cM$ at $[Z]$ is canonically given by
$$ \cM([Z]) = \bigotimes_P \cL(P)^{n_P}. $$
\end{remark}

Now we consider the GIT quotient of $\SymXC$ by the group $G[n]$. For any closed point $[Z] \in \SymXC$, we say $Z$ has \emph{smooth support} if the condition in \cite[Definition 2.5]{GHH-2019} is satisfied. In such a case, we can define its \emph{numerical support} $\vv(Z)$ as in \cite[Definition 2.6]{GHH-2019} and its \emph{combinatorial support} $\vv_{\mathbf{a}}$ as in \cite[Section 2.3.3]{GHH-2019}.

The GIT analysis in \cite[Section 2]{GHH-2019} applies literally here. Parallel to \cite[Theorem 2.10]{GHH-2019}, we can conclude the following description of the GIT (semi-)stable locus:

\begin{proposition}
\label{prop:semi}
With respect to the $G[n]$-linearized line bundle $\cM$, we have
\begin{align*}
&\quad\ \SymXC^{ss}(\cM) = \SymXC^{s}(\cM) \\
&= \left\{ [Z] \in \SymXC \ \middle| \begin{array}{c} Z \text{ has smooth support} \\ \text{and } \vv(Z) = \vv_{\mathbf{a}} \end{array} \right\}.
\end{align*}
\end{proposition}

\begin{proof}
The same as \cite[Theorem 2.10]{GHH-2019}.
\end{proof}

Now we consider the following relative Hilbert-Chow morphism; see e.g. \cite[Paper III, Section 4.3]{Rydh}:
\begin{equation}
\label{eqn:HC}
\pi: \HilbXC \to \SymXC.
\end{equation}
It is clear that $\pi$ respects the $G[n]$-actions. Combining Proposition \ref{prop:semi} and \cite[Theorem 2.10]{GHH-2019} we have

\begin{corollary}
\label{cor:redo}
For any sufficiently large $\ell$, we have
\begin{align*}
&\quad\ \HilbXC^{ss}(\cM_\ell) \\
&= \HilbXC^{s}(\cM_\ell) \\
&= \pi^{-1}(\SymXC^{ss}(\cM)) \\
&= \pi^{-1}(\SymXC^{s}(\cM)) \\
&= \left\{ [Z] \in \HilbXC \ \middle| \begin{array}{c} Z \text{ has smooth support} \\ \text{and } \vv(Z) = \vv_{\mathrm{a}} \end{array} \right\}. \qed
\end{align*}
\end{corollary}

\subsection{Semi-stable loci of the $n$-fold product}

To determine the semi\--stable locus of the $n$-fold product $\PrXC$, we utilize the following general result which states that the semi-stable locus is functorial with respect to finite group quotient.

\begin{lemma}\label{lem:stablequotient}
\label{lem:quot}
Let $W$ be a quasi-projective scheme over an algebraically closed field $\Bbbk$ of characteristic zero. Let $H$ be a finite group acting on $W$, and $f: W \to V$ the quotient morphism. Let $G$ be a reductive group acting on $W$, which commutes with the $H$-action on $W$ and induces a $G$-action on $V$. Let $L$ be a $G$-linearized ample line bundle on $V$. Then we have
\begin{align*}
W^{ss}(f^*L) &= f^{-1}(V^{ss}(L)); \\
W^s(f^*L) &= f^{-1}(V^s(L)).
\end{align*}
\end{lemma}

\begin{proof}
We first prove the statement for semi-stable loci. For any closed point $w \in W$, let $v=f(w) \in V$. We need to show that $w$ is semi-stable if and only if $v$ is semi-stable.

Assume $v$ is semi-stable. Then there exists $s \in \Gamma(V, L^{\otimes n})^G$ for some positive integer $n$, such that $s(v) \neq 0$. Then we have $f^*s \in \Gamma(W, f^*L^{\otimes n})^G$ and $(f^*s)(w) \neq 0$, hence $w$ is semi-stable.

Assume $w$ is semi-stable. Then there exists $t \in \Gamma(W, f^*L^{\otimes n})^G$, such that $t(w) \neq 0$. Then for every positive integer $m$, the section
$$ t_m = \sum_{h \in H}(h^*(t))^m \in \Gamma(W, f^*L^{\otimes mn})^G $$
is also $H$-invariant, and descends to a section $s_m \in \Gamma(V, L^{\otimes mn})^G$. Since $t(w) \neq 0$, we have $t_{m_0}(w) \neq 0$ for some positive integer $m_0$. It follows that $s_{m_0}(v) \neq 0$, hence $v$ is semi-stable.

We next show the statement holds for stable loci. Notice that $G_w$ is a subgroup of $G_v$, and that the quotient $G_v/G_w$ is bijective to the set 
$$ \{ w' \in W \mid w' \in G \cdot w \text{ and } f(w')=v \}. $$
It follows that $|G_v/G_w| < |H|$ is finite, hence $G_v$ is finite if and only if $G_w$ is finite.

Moreover, let $f^{ss}$ be the restriction of $f$ to $W^{ss}(f^*L)$. Then $f^{ss}$ is also given by a finite group quotient, hence a finite map. If the orbit $G \cdot w$ is closed in $W^{ss}(f^*L)$, then its image $G \cdot v$ is also closed in $V^{ss}(L)$. If the orbit $G \cdot w$ is not closed, then there is another semi-stable orbit $G \cdot w'$ in the closure of $G \cdot w$. Let $G \cdot v'$ be the image of $G \cdot w'$ under $f^{ss}$, then $G \cdot v'$ is also in the closure of $G \cdot v$. If $G \cdot v$ were closed, we would have $G \cdot v = G \cdot v'$ which is a contradiction, as the dimension of $G \cdot v'$ is strictly smaller than the dimension of $G \cdot v$. We conclude that $G \cdot w$ is closed if and only if $G \cdot v$ is closed, hence $w$ is stable if and only if $v$ is stable.
\end{proof}

To apply Lemma \ref{lem:quot}, recall from Lemma \ref{lem:ampcon} that we have a $G[n]$-linearized relatively ample line bundle $\cL^{\boxtimes n}$ on $\PrXC$, which descends to a $G[n]$-linearized 
relatively ample line bundle $\cM$ on $\SymXC$.

\begin{corollary}\label{cor:stablequotient}
\label{cor:prod}
We have
\begin{align*}
&\quad\ (\PrXC)^{ss}(\cL^{\boxtimes n}) \\
&= (\PrXC)^{s}(\cL^{\boxtimes n}) \\
&= \tau^{-1}(\SymXC^{ss}(\cM)) \\
&= \tau^{-1}(\SymXC^{s}(\cM)) \\
&= \left\{ [Z] \in \PrXC \ \middle| \begin{array}{c} Z \text{ has smooth support} \\ \text{and } \vv(\tau(Z)) = \vv_{\mathrm{a}} \end{array} \right\}.
\end{align*}
\end{corollary}

\begin{proof}
This follows immediately from Lemma \ref{lem:quot} and Proposition \ref{prop:semi}.
\end{proof}

\begin{remark}
	Indeed, we could also perform a GIT analysis that is similar to \cite[Section 2]{GHH-2019} to obtain the (semi-)stable locus on the $n$-fold product $\PrXC$, which would give a different proof of Corollary \ref{cor:stablequotient} without using Lemma \ref{lem:quot}.
\end{remark}

\begin{remark}
	We point out that Nagai has made the same observation in \cite[Remark 4.4.1]{nagai-2017} that the GIT analysis in \cite[Theorem 2.10]{GHH-2019} also works for the families $\SymXC$ and $\PrXC$ over $C[n]$.
\end{remark}


\section{The strata of the degenerate fibre}\label{sec:strata}

\subsection{GIT quotients of three families}

Let $ f \colon X \to C $ be a strict simple degeneration and fix an integer $ n > 0 $. This gives us an expanded degeneration $ f[n] \colon X[n] \to C[n] $,  which in turn gives rise to the following three families 
related by the quotient morphism $\tau$, see \eqref{eqn:tau} and the Hilbert-Chow morphism $\pi$, see \eqref{eqn:HC}
\begin{equation}
\label{eqn:noss}
\xymatrix{
\PrXC \ar[rrd] \ar[d]_{\tau} && \\
\SymXC \ar[rr] && C[n] \\
\HilbXC \ar[rru] \ar[u]^{\pi}. && 
}
\end{equation}

All three families come equipped with $G[n]$-actions, which can be lifted to some natural line bundles to determine the corresponding (semi-)stable loci. For simplicity we omit the references to the line bundles in the notation of (semi-)stable loci from now on. By Corollary \ref{cor:redo} and Corollary \ref{cor:prod}, we know that the Hilbert-Chow morphism $\pi$ and the quotient morphism $\tau$ respect GIT stability. Therefore we have the
following commutative diagram of  semi-stable loci
\begin{equation}
\label{eqn:bigcomm}
\xymatrix{
(\PrXC)^{ss} \ar[rrd]^-{\varphi_p} \ar[d]_{\tau^{ss}} && \\
\SymXC^{ss} \ar[rr]^-{\varphi_s} && C[n] \\
\HilbXC^{ss} \ar[rru]_-{\varphi_h} \ar[u]^{\pi^{ss}}. && 
}
\end{equation}

By Proposition \ref{prop:semi} or Corollary \ref{cor:redo}, we see that a semi-stable point in $\HilbXC$ or $\SymXC$ always has smooth support. We shall assume, unless explicitly stated otherwise, that the original family $f:X \to C$ has relative dimension at most $2$. Then we have
\begin{itemize}
\item $\HilbXC^{ss}$ is a smooth variety and $\varphi_h$ is a smooth morphism;
\item $\SymXC^{ss}$ is singular along the diagonal (when $n_P \geqslant 2$ for some $P$), and $\pi^{ss}$ is a divisorial resolution.
\end{itemize}

All arrows in  diagram \eqref{eqn:bigcomm} are equivariant with respect to $G[n]$-actions. Passing to the $G[n]$-quotients, we obtain another commutative diagramme of the quotients related by natural maps
\begin{equation}
\label{eqn:smcomm}
\xymatrix{
P^n_{X/C} := (\PrXC)^{ss}/G[n] \ar[rrd]^-{\psi_p} \ar[d]_{\xi} && \\
J^n_{X/C} := \SymXC^{ss}/G[n] \ar[rr]^<<<<<<{\psi_s} && C \\
I^n_{X/C} := \HilbXC^{ss}/G[n] \ar[rru]_-{\psi_h} \ar[u]^{\eta}. && 
}
\end{equation}

\begin{remark}
\label{rmk:PtoJ}
We notice that the $G[n]$-action and the $\mathfrak{S}_n$-action on $\PrXC$ commute, from which we conclude immediately that $J^n_{X/C}$ is the $\mathfrak{S}_n$-quotient of $P^n_{X/C}$ by Corollary \ref{cor:prod}.
\end{remark}

Our primary goal in the next  sections is to study the finer structure of the three families $P^n_{X/C}$, $J^n_{X/C}$ and $I^n_{X/C}$. We will focus on the birational geometry of the three families, the geometry and combinatorics of the degenerate fibres, as well as understand how these structures are related via the natural maps $\xi$ and $\eta$. After recalling the definition of the dual complex of the degenerate fibre, we will pursue the first key step in this section, namely producing the stratification and studying the restriction relations among components.

\subsection{Dual complexes and other preliminaries}\label{subsec:dualcomplexdef}

Let $ E $ be a smooth $\kk$-variety, and let $E_0 = \cup_{a \in A} E_a $ be a (reduced) strict normal crossing divisor on $E$ with  components $E_a$. We recall the well-known definition of the \emph{dual complex} $\Delta(E_0)$ attached to $E_0$. This is a combinatorial object which encodes how the components $E_a$ intersect.

\begin{definition}
\label{def:1st-dual}
The dual complex $\Delta(E_0)$ of a strict normal crossings divisor $E_0$ is the unique $\Delta$-complex with the following properties:
\begin{itemize}
\item[\rm{(i)}]  The $d$-dimensional simplices correspond bijectively to the connected components $E_B^i$ of $ E_B = \cap_{b \in B} E_b $, as $B$ runs through the 
subsets $ \emptyset \neq B \subset A $ with $ \vert B \vert = d+1 $.

\item[\rm{(ii)}]  Let $ B$ and $B'$ be two non-empty subsets of $A$. Then $\mathrm{Simp}(E_B^i)$ is a face of $\mathrm{Simp}(E_{B'}^{i'})$ if and only if $ E_{B'}^{i'} \subset  E_B^i $.

\end{itemize}
\end{definition}

In fact, the notion of the dual complex can be defined in a much more general set-up. We recall \cite[Definition 8]{dFKX-2017}, which will be used in our later discussion.

\begin{definition}[{\cite[Definition 8]{dFKX-2017}}]
\label{def:2nd-dual}
Let $E_0 = \cup_{a \in A} E_a$ be a pure dimensional scheme with irreducible components $E_a$. Assume that
\begin{itemize}
\item[\rm{(i)}]  each $E_a$ is normal, and
\item[\rm{(ii)}]  for each $B \subseteq A$, if $\cap_{b \in B}E_a$ is non-empty, then every connected component of $\cap_{b \in B}E_a$ is irreducible and has codimension $|B|-1$ in $E_0$.
\end{itemize}
Then the dual complex $\Delta(E_0)$ can be defined as in Definition \ref{def:1st-dual}.
\end{definition}

In the next few sections, we will be particularly interested in the dual complexes of the degenerate fibres of the three families $P^n_{X/C}$, $J^n_{X/C}$ and $I^n_{X/C}$. We start with some foundational analysis of the intersection relation of various strata in the degenerate fibre of the third family.

\subsection{The Hilbert scheme}\label{subsec:Hilbertscheme}

Consider the morphism
$$ \varphi_h \colon \HilbXC^{ss} \to C[n]. $$
As noted above, this map is smooth by our assumption that $ X \to C $ has relative dimension at most $2$. Composition with the natural map $ C[n] \to C$ yields a morphism
$$ \phi_h \colon \HilbXC^{ss} \to C $$
where the special fibre $ \phi_h^{-1}(0) $ forms a strict normal crossings divisor on the smooth variety $\HilbXC^{ss}$. This is immediate from the fact that $ C[n] \to C $ is obtained as a pullback of $ \mathbb{A}^{n+1} \to \mathbb{A}^1 $ along the \'etale map $ C \to \mathbb{A}^1$.

To ease notation in the following discussion, we shall write $ \mathcal{H} $ for $ \HilbXC^{ss} $. Our aim is now to study the geometry of the components of $ \mathcal{H}_0 = \phi_h^{-1}(0) $, and how they intersect. This analysis forms the foundation of many of the results later on in this paper. At this stage, we also remark that, as describing $ \mathcal{H}_0 $ only involves computations over $ 0 \in C $, and since $C[n] \to \mathbb{A}^{n+1}$ restricts to an isomorphism over the special fibre $ (\mathbb{A}^{n+1})_0 \to 0 \in \mathbb{A}^1 $, we may assume that $ C = \mathbb{A}^1 $ from the beginning. Consequently, we shall replace $C[n]$ by $\mathbb{A}^{n+1}$ in our notation.

\subsubsection{} 
\label{subsubsec:notation}
We first introduce some notation. 

\begin{itemize}
\item If $n$ is understood from the context, we shall often write $ \mathbb{A} $ instead of $ \mathbb{A}^{n+1} $, to simplify notation. 

\item For any subset $ \varnothing \neq I  \subset [n+1] $, we put $\mathbb{A}_I = V(t_i \mid i \in I)$ and 
$$ U_I = \{ (t_1, \ldots, t_{n+1}) \mid t_i = 0 \text{ for all } i \in I,\ t_j \neq 0 \text{ for all } j \notin I \}. $$
We remark that $U_I$ is open in $\mathbb{A}_I$.  

\item Let $ I^c $ denote the complement of $I$ in $[n+1]$. Let moreover $ j \in I^c $, and put $ I_j = I \cup \{j\} $. Then we define 
$$U_{I,j} =  \{ (t_1, \ldots, t_{n+1}) \mid t_i = 0 \text{ for all } i \in I,\ t_{i'} \neq 0 \text{ for all } i' \notin I_j \}. $$ 
The set $ U_{I,j} $ forms a partial compactification of $ U_I $ inside $\mathbb{A}_I$, where we allow also the $j$-th coordinate to be zero. Observe that $ U_{I,j} \setminus U_I = U_{I_j} $.

\item Let $ W $ be a scheme over $\mathbb{A}_I$, with irreducible components $ W_{\alpha}$, $ \alpha \in \mathcal{A} $. Then we denote by $ W^{\circ} $ the pullback of $W$ to $U_I$. For each $\alpha$, we moreover put 
$$ W_{\alpha}^* = W_{\alpha}^{\circ} \setminus \cup_{\beta \neq \alpha} W_{\beta}^{\circ}. $$
\end{itemize}

\subsubsection{}
By smoothness of $\varphi_h \colon \mathcal{H} \to \mathbb{A}^{n+1} $ it follows for each $ I $, that   
$$\mathcal{H}_I = \varphi_h^{-1}(\mathbb{A}_I)$$ 
is a disjoint union of smooth irreducible components (recall that $\mathbb{A}_I$ is smooth and irreducible). As it turns out, the components of $\mathcal{H}_I$ are somewhat difficult to describe directly. However, by flatness of the restriction 
$$ \varphi_h \vert_{\mathbb{A}_I} \colon \mathcal{H}_I \to \mathbb{A}_I, $$ 
every component $\mathcal{C}$ of $\mathcal{H}_I$ has dense image in $\mathbb{A}_I$. Thus, each $\mathcal{C}$ equals the closure in $\mathcal{H}_I$ of a unique component $\mathcal{C}^{\circ}$ of $\mathcal{H}_I^{\circ}$. Below, we shall give a useful description of the components of $\mathcal{H}_I^{\circ}$, which turn out to be quite easy to describe.

\subsubsection{}\label{subsubsec:hilbind}
To index the components of $\mathcal{H}_I^{\circ}$, where $ I = \{i_1, \ldots, i_r\} \subset [n+1] $, we will use the the dual graph $ \Gamma = (V,E)$ attached to the degeneration $ X \to C $. Having fixed a bipartite oriention of $\Gamma$, we write $V^+$, resp.~$V^-$, for the set of vertices with only outgoing and incoming arrows respectively. For each $v \in V$, let moreover $E(v)^+$ denote the set of edges directed towards $v$, and let $E(v)^-$ denote the set of edges directed away from $v$.

We introduce the following additional notation. 

\begin{itemize}
\item $ \mathbf{b} = \{b_v\}_{v \in V} $ denotes a collection of non-negative integers.  
\item $ \mathbf{s} = \{\mathbf{s}_{\gamma}\}_{\gamma \in E} $ denotes a collection of (ordered) tuples 
$$ \mathbf{s}_{\gamma} = (s_{\gamma,1}, \ldots, s_{\gamma, r-1}) $$ 
where $ s_{\gamma,l} \in \mathbb{Z}_{\geqslant 0} $. 
\end{itemize}

\begin{definition}
\label{def:num-stable}
We say that the pair $ (\mathbf{b}, \mathbf{s}) $ is \emph{stable} with respect to $I$ if the vector 
$$ \mathbf{v}(\mathbf{b}, \mathbf{s}) = (\sum_{v \in V^+} b_v, \sum_{\gamma \in E} s_{\gamma,1}, \ldots, \sum_{\gamma \in E} s_{\gamma,r-1}, \sum_{v \in V^-} b_v ) \in \mathbb{Z}^{r+1} $$
is stable in the sense of Theorem \ref{thm:GIT-main}, i.e., if $\mathbf{v}(\mathbf{b}, \mathbf{s}) = \mathbf{v}_{\mathbf{a}}$ where $\mathbf{a}$ is determined by $I$.
\end{definition}

\begin{remark}
If $ (\mathbf{b}, \mathbf{s}) $ is stable with respect to $I$, it is a consequence of Definition \ref{def:num-stable} that $ \sum_v b_v + \sum_{\gamma} \sum_l s_{\gamma, l} = n$.
\end{remark}

\begin{proposition}\label{prop:strat-comp}
The irreducible components of $ \mathcal{H}_I^{\circ}$ are indexed precisely by the pairs $(\mathbf{b},\mathbf{s})$ that are stable with respect to $I$. The component corresponding to $(\mathbf{b},\mathbf{s})$ is
$$ (\mathcal{H}_I)_{(\mathbf{b},\mathbf{s})}^{\circ} = \prod_{v \in V} \mathrm{Hilb}^{b_v}((Y_v)_I^{*}/U_I) \times \prod_{\gamma \in E} \prod_{l=1}^{r-1} \mathrm{Hilb}^{s_{\gamma,l}}((\Delta_I^{\gamma, i_l})^{*}/U_I), $$
where the products are fibred products over $U_I$.
\end{proposition}

\begin{proof}
For each $ i_l \in I \cup \{0\} $, $(\Delta_I^{\gamma, i_l})^{\circ}$ is \emph{smooth} over $U_I$, with irreducible fibres of dimension $\leqslant 2$. It follows that the non-smooth locus of $X[n]_I^{\circ}$ is located precisely where the components intersect. In other words, the smooth locus $ (X[n]_I^{\circ})^{sm} $ of $X[n]_I^{\circ} \to U_I  $ is a disjoint union of the components $(\Delta_I^{\gamma, i_l})^*$. 

Thus $\mathrm{Hilb}^n((X[n]_I^{\circ})^{sm}/U_I) \cap \mathcal{H}$ is a disjoint union of schemes of the form $ (\mathcal{H}_I)_{(\mathbf{b},\mathbf{s})}^{\circ} $, where $(\mathbf{b},\mathbf{s})$ runs over pairs that are stable with respect to $I$. It is straightforward to verify that each $(\mathcal{H}_I)_{(\mathbf{b},\mathbf{s})}^{\circ}$ is irreducible. 
\end{proof}

\subsection{Stratification of the components}
For each $(\mathbf{b},\mathbf{s})$, we denote by $ (\mathcal{H}_I)_{(\mathbf{b},\mathbf{s})} $ the closure of $ (\mathcal{H}_I)_{(\mathbf{b},\mathbf{s})}^{\circ} $ in $\mathcal{H}_I$. We will next study the boundary 
$$ (\mathcal{H}_I)_{(\mathbf{b},\mathbf{s})} \setminus (\mathcal{H}_I)_{(\mathbf{b},\mathbf{s})}^{\circ}, $$ 
by carefully investigating the strata of $ (\mathcal{H}_I)_{(\mathbf{b},\mathbf{s})} $ over $ U_{I,j} \setminus U_I $, where $ j \in I^c $. In light of our inductive set-up, these strata are again of the form $ (\mathcal{H}_{I_j})_{(\mathbf{b}',\mathbf{s}')}^{\circ} $, for certain pairs $(\mathbf{b}',\mathbf{s}')$ stable with respect to $I_j$.  

We still need to compute the pairs $(\mathbf{b}',\mathbf{s}')$ that can occur as the index of a boundary stratum. For this purpose, the following easy lemma will be quite useful; it tells us that it suffices to produce a point in the intersection of $ (\mathcal{H}_I)_{(\mathbf{b},\mathbf{s})} $ and $ (\mathcal{H}_{I_j})_{(\mathbf{b}',\mathbf{s}')}^{\circ} $, in order to show containment.

\begin{lemma}\label{lemma:containment}
A component $ (\mathcal{H}_{I_j}^{\circ})_{(\mathbf{b}',\mathbf{s}')} $ of $\mathcal{H}_{I_j}^{\circ}$ is a subscheme of $ (\mathcal{H}_I)_{(\mathbf{b},\mathbf{s})} $ if and only if
$$ (\mathcal{H}_{I_j}^{\circ})_{(\mathbf{b}',\mathbf{s}')} \cap  (\mathcal{H}_I)_{(\mathbf{b},\mathbf{s})}  \neq \emptyset. $$
\end{lemma}

\begin{proof}
Note that $\mathcal{H}_{I_j}^{\circ}$ is a subscheme of $ \mathcal{H}_I $. Moreover, they are both disjoint unions of their irreducible components. Hence, $ (\mathcal{H}_{I_j})_{(\mathbf{b}',\mathbf{s}')}^{\circ} $ is a subscheme of a unique component $ (\mathcal{H}_I)_{(\mathbf{b},\mathbf{s})} $, and this happens if and only if they have non-empty intersection.
\end{proof}

\subsubsection{}
We are now ready to formulate the main result of this paragraph, but first we need to introduce some additional notation.  We fix a subset $ J = \{j_1, \ldots, j_{r+1}\} $ in $ [n+1]$. It has $r+1$ subsets of cardinality $r$, denoted 
$$ I(k) = J \setminus \{j_k\}. $$
We also fix $ \mathbf{b}' = \{b'_v\}_{v \in V} $ and $ \mathbf{s}' = \{s'_{\gamma}\}_{\gamma \in E} $, where $ s'_{\gamma} = (s'_{\gamma, 1}, \ldots, s'_{\gamma, r}) $. We assume that $(\mathbf{b}', \mathbf{s}')$ is stable with respect to $J$.

\begin{theorem}
\label{thm:limit-scheme}
For each $k$, $ (\mathcal{H}_J^{\circ})_{(\mathbf{b}', \mathbf{s}')} $ is contained in the closure of a unique component over $U_{I(k)} $, denoted $ (\mathcal{H}_{I(k)}^{\circ})_{(\mathbf{b}, \mathbf{s})} $. Both $ \mathbf{b} $ and $ \mathbf{s} $ depend explicitly on $k$, and can be described as follows:

\begin{align*}
k=1\colon 
&\begin{cases}
b_v = b'_v + \sum_{\gamma\in E(v)^-} s'_{\gamma,1}\\
s_{\gamma} = (s'_{\gamma,2},\dots,s'_{\gamma,r})
\end{cases}\\
1<k<r+1\colon
&\begin{cases}
b_v = b_v'\\
s_\gamma = (s'_{\gamma,1},\dots, s'_{\gamma,k-1}+s'_{\gamma,k}\dots, s'_{\gamma,r})
\end{cases}\\
k=r+1\colon
&\begin{cases}
b_v = b'_v + \sum_{\gamma\in E(v)^+} s'_{\gamma,r}\\
s_\gamma = (s'_{\gamma,1},\dots,s'_{\gamma,r-1}).
\end{cases}
\end{align*}

\end{theorem}

\begin{proof}
We fix $ k $ and consider $ I = I(k) = J \setminus \{j_k\} $. For simplicity, we assume $ 1 < k < r + 1 $; the cases $ k = 1, r + 1 $ are entirely similar, but require slight modifications in notation. Recall moreover that $ U_J = U_{I,j_k} \setminus U_I $.   

Pick a closed point $ q \in U_J $. Then we can find a morphism 
$$ S = \mathrm{Spec}~k[[\pi]] \to U_{I,j_k}, $$
with $\pi $ a formal variable, sending the closed point $s \in S$ to $q$, and the generic point $\eta \in S$ to $ U_I $. We define $ \varphi $ as the composition
$$ \varphi \colon S \to U_{I,j_k} \subset \mathbb{A}_I. $$
If $W$ is a scheme over $ \mathbb{A}_I $, we write $ \varphi^* W  $ for its pullback along $\varphi$.

We also fix an edge $ \gamma \in E $ (this $\gamma$ will be suppressed in the notation). For any $ j_l \in I $, we denote by $ \tilde{\Delta}_I^{j_l} \subset \Delta_I^{j_l} $ the open subscheme obtained by removing from $ \Delta_I^{j_l} $ its intersection with \emph{all} neighbouring components in $X[n]_I$. By \cite[Prop.~1.12]{GHH-2019} (which in particular describes the degeneration of $\Delta_I^{j_l}$ as $t_{j_k}$ tends to zero), we find that $ \varphi^*\tilde{\Delta}^{j_l} \to S$ is smooth if $ l \neq k-1 $, and if $ l = k-1 $, the generic fibre is smooth, and degenerates into a normal crossing union $ (\Delta_{J}^{j_{k-1}})_s \cup (\Delta_{J}^{j_k})_s $ over $ s \in S$.

In either case, let $T_s$ be a point in the smooth locus of $ (\varphi^*\tilde{\Delta}_I^{j_l})_s $. By \cite[Cor.~6.2.13]{liu-2002}, $T_{s}$ can be lifted to an $S$-section $T$ of $\varphi^*\tilde{\Delta}_I^{i_l}$. Clearly, $T$ can also be viewed as a section of $ \varphi^* X[n]_I \to S $. By construction, it is contained in the smooth locus.

Now let $ T_s $ be a disjoint union of $n$ points contained in the smooth locus of $ (\varphi^*X[n]_I)_s $. By the above procedure, we obtain $n$ pairwise disjoint sections $ T_i $ by lifting each point in the support of $ T_s $. Then we define $T = \cup_i T_i $. Being a disjoint union of sections, it is a closed subscheme of $ \varphi^* X[n]_I $, and flat over $S$.

Assume that $ T_s $ is GIT stable, belonging to a component $ (\mathcal{H}_J^{\circ})_{(\mathbf{b}', \mathbf{s}')} $ over $U_J$. Using the explicit construction of $T$, it is straightforward to verify that the generic fibre $ T_{\eta} $ is contained in the component $ (\mathcal{H}_{I}^{\circ})_{(\mathbf{b}, \mathbf{s})} $, where $(\mathbf{b}, \mathbf{s})$ is stable with respect to $I$, and depends on $k$ as specified above. Hence $ T_s $ belongs to the closure $ (\mathcal{H}_{I})_{(\mathbf{b}, \mathbf{s})} $, as it is a specialization of $ T_{\eta} $. Then Lemma \ref{lemma:containment} asserts that $ (\mathcal{H}_J^{\circ})_{(\mathbf{b}', \mathbf{s}')} $ is a subscheme of $ (\mathcal{H}_{I})_{(\mathbf{b}, \mathbf{s})} $.
\end{proof}

\subsubsection{}\label{subsubsec:prod-comp} We also need analogues of Proposition \ref{prop:strat-comp} and Theorem \ref{thm:limit-scheme} for the $n$-fold product and the symmetric product. With a few modifications, this follows along the lines of the proof in the Hilbert scheme case, hence we will be particularly brief on details in this paragraph.

We keep $ n \geqslant 1 $ fixed, and denote the stable loci by $ \mathcal{P} = (X[n] \times_{C[n]} \cdots \times_{C[n]} X[n])^{ss} $ and $ \mathcal{S} = \mathrm{Sym}^n(X[n]/C[n])^{ss} $; both $ \mathcal{P} $ and $ \mathcal{S} $ are flat over $ C[n] $. By Corollary \ref{cor:prod}, $ \mathcal{P} $ is also \emph{smooth} over $C[n]$, and there is an $\mathfrak{S}_n$-quotient map
$$ \tau^{ss} \colon \mathcal{P} \to \mathcal{S}. $$
For each non-empty subset $ I = \{i_1, \ldots, i_r \} \subset [n+1] $, $ \mathcal{P}_I $ is a disjoint union of its (smooth) irreducible components. Moreover, $ \mathcal{S}_I $ is an $\mathfrak{S}_n$-quotient of the smooth quasi-projective $\kk$-variety $ \mathcal{P}_I $, hence it is normal. This implies that $ \mathcal{S}_I $ is a disjoint union of its (normal) irreducible components, as well.

With these preliminaries at hand, it is entirely straightforward to treat the symmetric product. One finds that the irreducible components of $ \mathcal{S}_I^{\circ} $ are indexed precisely by pairs $(\mathbf{b},\mathbf{s})$ (as introduced in Paragraph \ref{subsubsec:hilbind}) that are stable with respect to $I$. Each component $ (\mathcal{S}_I)^{\circ}_{(\mathbf{b},\mathbf{s})} $ can be computed in an analogous way as in Proposition \ref{prop:strat-comp} (the proof is, word for word, the same). Also Theorem \ref{thm:limit-scheme}, both the statement and the proof, transfers immediately to the symmetric product case; we leave the details to the reader. 

For the $n$-fold product, we need to be more precise about the notation. Recall from Section \ref{sec:GIT} that the choice of $ I \subset [n+1] $ yields an expansion $ \Gamma_I $ of $\Gamma$. We denote the \emph{black} vertices by $ v_I $, where $ v \in V $, and the \emph{white} vertices by $ (I, \gamma, i_l) $, where $ \gamma \in E $, and where $ 1 \leqslant l \leqslant r-1 $. Then the components of $ \mathcal{P}_I^{\circ} $ are indexed by \emph{stable} tuples, denoted $ \mathbf{z} $, in the $n$-fold product
$$ V(\Gamma_I) \times \cdots \times V(\Gamma_I), $$
where $ V(\Gamma_I) $ is the set of \emph{all} vertices of $\Gamma_I$. Here stability is formulated as follows: 
\begin{itemize}
\item For each $ v \in V $, define $ b_v $ as the number of times $v_I $ occurs as an entry in $ \mathbf{z} $.
\item For each $ \gamma \in E$ and $l$, define $ s_{\gamma, l} $ as the number of times $ (I, \gamma, i_l) $ occurs as an entry in $ \mathbf{z} $. 
\item Define the numerical data $(\mathbf{b},\mathbf{s})$ as in Paragraph \ref{subsubsec:hilbind}. We say that $ \mathbf{z} $ is stable with respect to $I$ if and only if $\textbf{v}(\mathbf{b},\mathbf{s})$ is stable with respect to $I$, in the sense of Definition \ref{def:num-stable}
\end{itemize}
The analogue of Proposition \ref{prop:strat-comp} can now be easily formulated (left to the reader). One finds that each component $ (\mathcal{P}_I)^{\circ}_{\textbf{z}} $ is a fibred product, over $U_I$, of schemes of the form $(Y_v)_I^{*}$, resp.~$(\Delta_I^{\gamma, i_l})^{*}$, dictated by $\textbf{z}$. 

For the applications in Section \ref{sec:dual-computation}, it is however necessary to give a detailed analogue of Theorem \ref{thm:limit-scheme}. Let $ J = \{j_1, \ldots, j_{r+1} \} \subset  [n+1] $, let $ I(k) = J \setminus \{j_k\} $, and let $ \mathbf{z}' \in V(\Gamma_J) \times \cdots \times V(\Gamma_J) $ be a stable $n$-tuple (with respect to $J$). Then, by the same method of proof as in Theorem \ref{thm:limit-scheme}, one finds: 

\begin{theorem}\label{thm:limit-scheme-prod}
For each $k$, $ (\mathcal{P}_J)^{\circ}_{\textbf{z}'} $ is contained in the closure of a unique component over $U_{I(k)}$, denoted $ (\mathcal{P}_{I(k)})^{\circ}_{\textbf{z}} $. The tuple $\textbf{z}$ depends explicitly on $k$, and is obtained from $\textbf{z}'$ as follows: 
\begin{itemize}
\item[\rm{(i)}]  Every vertex of the form $v_J$ appearing as an entry in $\textbf{z}'$ is replaced by $v_{I(k)}$ in the same entry. For the remaining entries, we use the rules listed below. 

\item[\rm{(ii)}]  $k=1$: $(J, \gamma, j_1)$ is replaced by $v_{I(k)}$, where $ \gamma \in E(v)^-$; $ (J,\gamma,j_l) $ is replaced by $(I(k),\gamma,j_l)$ otherwise.

\item[\rm{(iii)}]  $1 < k < r+1$: $ (J,\gamma,j_l) $ is replaced by $(I(k),\gamma,j_l)$ if $ l \neq k$; $ (J,\gamma,j_k) $ is replaced by $(I(k),\gamma,j_{k-1})$.

\item[\rm{(iv)}]  $k=r+1$: $ (J,\gamma,j_l) $ is replaced by $(I(k),\gamma,j_l)$ if $ l \neq r$; $ (J,\gamma,j_{r}) $ is replaced by $v_{I(k)}$, where $ \gamma \in E(v)^+$.
\end{itemize}

\end{theorem}


\section{The DLT structure}\label{sec:dlt}

\subsection{DLT models}\label{sec:dlt models}

Our aim here is to show that $I^n_{X/C}\to C$ is a \emph{dlt model}. This
terminology was explained in Section \ref{sec:intro-mainresults}.

\subsubsection{}

Let $X\to C$ be a projective strict simple degeneration of relative dimension  at most
two such that the dual graph has no odd cycles. The reason why we restrict to
fibre dimension at most $2$ is that otherwise the Hilbert schemes are in general
neither irreducible nor equi-dimensional. 

By construction $I^n_{X/C}\to C$ is a good quotient of the family $\mathcal{H}
\to C[n]$ by the $n$-dimensional torus $G[n]$. Here $\mathcal{H}\subset
\Hilb^n(X[n]/C[n])$ is the GIT stable locus with respect to a certain
linearization and $C[n]$ is the fibre product $C\times_{\AA^1}\AA^{n+1}$ with
respect to a chosen \'etale coordinate $t\colon C\to \AA^1$ and the $(n+1)$-fold
multiplication $\AA^{n+1}\to \AA^1$. By a \emph{coordinate hyperplane} in $C[n]$
we mean the inverse image of a coordinate hyperplane in $\AA^{n+1}$.

\begin{proposition}\label{prop:Qfactorial}
$I^n_{X/C}$ is normal, has finite quotient singularities only, and is
$\QQ$-factorial.
\end{proposition}

\begin{proof}
A good quotient of a normal variety is normal, and our $I^n_{X/C}$ is even a
geometric quotient of the nonsingular variety $\mathcal{H}$. Moreover all
stabilizer groups are finite, so by Luna's \'etale slice theorem, $I^n_{X/C}$ is
\'etale locally a finite quotient. This in turn implies $\QQ$-factoriality
\cite[Prop.\ 5.15]{KM-1998} (the reference uses the language of  euclidean
topology over $\mathbb{C}$; the same argument works in  \'etale topology).
\end{proof}

\subsubsection{}

Our starting point for showing that $I^n_{X/C}\to C$ is a dlt model is that
$\mathcal{H}\to C$ is even an snc model. Here we view $\mathcal{H}$ as a family
over $C$ via the natural map $C[n]\to C$, thus the special fibre $\mathcal{H}_0$
is the inverse image of the union of all coordinate hyperplanes in $C[n]$.

\begin{lemma}\label{lem:snc}
$(\mathcal{H}, \mathcal{H}_0)$ is an snc pair.
\end{lemma}

\begin{proof}
$\mathcal{H}\to C[n]$ is smooth and $\mathcal{H}_0$ is the inverse
image of the strict normal crossing divisor which is the union of the coordinate
hyperplanes in $C[n]$. The claim follows from this.
\end{proof}

\subsubsection{}

Next we observe that there is a natural correspondence between the components of
the special fibres of $\mathcal{H}$ and $I^n_{X/C}$. This will be strengthened
below in Corollary \ref{cor:same-complex} to show that the associated dual complexes
can be identified.

\begin{lemma}\label{lem:component-bijection}
Let $\pi\colon \mathcal{H}\to I^n_{X/C}$ denote the geometric quotient map.
The rules $D = \pi(E)$ and $E = \pi^{-1}(D)$ define a bijection between
the components of $\mathcal{H}_0$ and of $(I^n_{X/C})_0$.
\end{lemma}

\begin{proof}
Since $\pi$ is a geometric quotient and each component $E\subset \mathcal{H}_0$
is invariant, the image $\pi(E)\subset I^n_{X/C}$ is closed, irreducible and has
codimension one, hence it is a component of $(I^n_{X/C})_0$. Conversely
$\pi^{-1}(D)$ is closed and has codimension one, so it is a union of components
$E_i\subset \mathcal{H}_0$. But the fibres of $\pi$ are exactly orbits, and each
$E_i$ is invariant, so there is room  for only one component $E=\pi^{-1}(D)$.
\end{proof}

\subsubsection{}

To deduce snc properties of $I^n_{X/C}$ from $\mathcal{H}$ we will apply Luna's
\'etale slice theorem. Let $G=G[n]$.

\begin{lemma}\label{lemma:luna}
Let $P\in \mathcal H$ be a point with stabilizer group $G_P\subset G$.
There exists a locally closed nonsingular $G_P$-invariant subvariety
$W_P\subset \mathcal H$ containing $P$, such that
\begin{enumerate}
\item[\rm{(i)}] the induced morphism $W_P/G_P \to I^n_{X/C}$ is \'etale,
\item[\rm{(ii)}]  the special fibre $t=0$ in $W_P$ is an snc divisor, more precisely $W_P \cap \mathcal{H}_0 =(W_P)_0=\cup_i F_i$, where $F_i=E_i\cap W_P$, and
\item[\rm{(iii)}]  $G_P$ acts set-theoretically freely outside the singular locus of $(W_P)_0$.
\end{enumerate}
\end{lemma}

\begin{proof}
This is an application of Luna's étale slice theorem (see \cite{SK-1989}
and \cite{GIT}), which works under the assumption that the orbit of $P$ is
closed. This holds for all orbits in our $\mathcal H$, as all points are GIT stable. 
Luna's theorem ensures the existence of a locally closed
$G_P$-invariant subvariety $W_P$, such that the induced morphism in (1) is
étale. Moreover, when $\mathcal H$ is nonsingular, the slice $W_P$ can be taken to
be nonsingular as well. The theorem says slightly more than this, but for
our purposes it suffices to note that, in the nonsingular case, the slice is
constructed by assuring that inside the tangent space $T_P(\mathcal H)$,
the tangent spaces to the slice $W_P$ and to the orbit $G\cdot P$ are
complementary.

Let $E_i\subset \mathcal H$ run through the set of components of the special fibre
$\mathcal{H}_0$ which pass through $P$. These are invariant under the $G$-action, so
the orbit $G\cdot P$ is contained in every $E_i$. In terms of tangent spaces, this
says that inside $T_P(\mathcal H)=T_P(W_P)\oplus T_P(G\cdot P)$ the
hyperplanes $T_P(E_i)$ are in general position and contain $T_P(W_P)$. It follows
that $T_P(E_i)=T_P(F_i)\oplus T_P(G\cdot P)$ where $F_i = E_i\cap W_P$ and thus
the tangent hyperplanes $T_P(F_i)$ inside $T_P(W_P)$ are in general position as
well. This says that $\cup_i F_i$ is snc at $P\in W_P$. As snc is an open
condition we obtain (2) after shrinking $W_P$ if necessary.

For (3) we note that $G$ acts freely on $\mathcal H$ outside the
singular locus of $\mathcal{H}_0$. This is so as $G$ acts freely outside the
singular locus of the coordinate hyperplanes in $\AA^{n+1}$, and $\mathcal{H}$
is smooth over $\AA^{n+1}$.
As the restriction of $\mathcal{H}_0$ to $W_P$ remains snc, with components $F_i
= E_i \cap W_P$, it follows that the singular locus of $\mathcal{H}_0$ restricts
to the singular locus of $(W_P)_0$.
\end{proof}

\subsubsection{}

The Luna slice Lemma above, combined with the observation that $(\mathcal{H},
\mathcal{H}_0)$ is snc, yields that only nontrivial stabilizer groups prevent
$(I^n_{X/C}, (I^n_{X/C})_0)$ from being snc:

\begin{lemma}\label{lem:snc-generically}
The pair $(I^n_{X/C}, (I^n_{X/C})_0)$ is snc around the image of every point $P\in
\mathcal{H}$ with trivial stabilizer.
\end{lemma}

\begin{proof}
Firstly $(\mathcal{H}, \mathcal{H}_0)$ is snc by Lemma \ref{lem:snc}.
When $P$ has trivial stabilizer the Luna slice $W_P\subset
\mathcal{H}$ in Lemma \ref{lemma:luna} has the property that $(W_P,
(W_P)_0)$ is normal crossing and $W_P \to I^n_{X/C}$ is \'etale, hence
$(I^n_{X/C}, (I^n_{X/C})_0)$ is normal crossing around the image of $P$.

For strictness, let $D\subset I^n_{X/C}$ be a component and
$E\subset\mathcal{H}_0$ the corresponding component as in Lemma
\ref{lem:component-bijection}. As $\mathcal{H}\to I^n_{X/C}$ is a universally geometric
quotient, $E\to D$ is a geometric quotient. For each point $P\in E$ with
trivial stabilizer there is a Luna slice through $P$ which is smooth over $\kk$
and \'etale over $D$. Thus $D$ is also smooth at the image of $P$.
\end{proof}

\begin{remark}
One can check that the nonsingular locus in $I^n_{X/C}$ coincides with the image
of points in $\mathcal{H}$ with trivial stabilizer. Thus by the lemma, the snc
locus is exactly the nonsingular locus. We do not need this fact, so we omit 
the details.
\end{remark}

\subsubsection{}

The next result, an application of Lemma \ref{lem:snc-generically}, is the main observation
assuring that $(I^n_{X/C}, (I^n_{X/C})_0)$ is ``sufficiently snc'' to have a
meaningful dual complex.

\begin{proposition}\label{prop:snc-almost-everywhere}
Let $Y\subset \bigcap_i D_i$ be a component of the intersection of a
given collection $\{D_i\}$ of the
components of the divisor $(I^n_{X/C})_0$. Then $Y$ has nontrivial intersection
with the snc locus of $(I^n_{X/C}, (I^n_{X/C})_0)$.
\end{proposition}

\begin{proof}
It suffices to find a point $P\in \mathcal{H}$ with trivial stabilizer and
mapping to $Y$. 
We will use the following easy observation as building block: when
$\Gm$ acts on a ruled variety $\Delta$ by scaling the fibres, there are finite
subschemes $Z\in \Hilb^n(\Delta)$ with trivial stabilizer group. If $n=1$ we can take any point, otherwise choose  
$n$ distinct points such that none is in the $\mu_m$-orbit of any other for all
$m\le n$.

By Proposition \ref{prop:strat-comp}, $Y$ is the image
of the closure of one of the loci $(\mathcal{H}_I)^\circ_{\mathbf{b}, \mathbf{s}}$, 
given explictly as a product of Hilbert schemes.
Concretely, this says that $P$ is given by a subscheme $Z$ in the nonsingular
locus of some fibre $X[n]_t$ over $t\in C[n]$, with the distribution of
$Z$ among the components of $X[n]_t$ dictated by $(\mathbf{b}, \mathbf{s})$:
$X[n]_t$ has a number of components $Y_v$ mapping isomorphically onto a
component of $X_0$ by the canonical map $X[n]\to X$, and $Z\cap Y_v$ is required
to have length
$b_v$. The remaining components of $X[n]_t$ are ruled varieties
$\Delta_{\gamma,l}$ and if $s_{\gamma,l}$ is the length of $Z\cap \Delta_{\gamma,\ell}$,
then stability requires that $\sum_{\gamma} s_{\gamma,l}>0$. 
The action by $G[n]$ is such that the stabilizer group of
the base point $t\in C[n]$ is a subtorus which can be written as a product of
$\Gm$'s indexed by $\gamma$, and for each $\gamma$ the corresponding $\Gm$ acts on
$X[n]_t$ by simultaneously scaling the fibres of the ruled varieties
$\Delta_{\gamma,l}$ for all $l$. By our initial observation a general such subscheme $Z$ has
trivial stabilizer.
\end{proof}

\begin{corollary}\label{cor:codimension}
For any collection $\{D_i\}$ of $k$ components of $(I^n_{X/C})_0$,
every component of the intersection $\bigcap_i D_i$ has codimension $k$ in
$I^n_{X/C}$.
\end{corollary}

\begin{proof}
By the proposition, every component of $\bigcap_i D_i$ intersects the snc locus of
$(I^n_{X/C}, (I^n_{X/C})_0)$ and the claim certainly holds for snc pairs.
\end{proof}

\subsubsection{}

We are now in the position to prove the result we have been aiming for.

\begin{theorem}\label{thm:dltmodel}
$I^n_{X/C}\to C$ is a dlt model.
\end{theorem}

\begin{proof}
For each $P\in \mathcal{H}$, let $W_P\subset \mathcal{H}$ be a Luna \'etale
slice as in Lemma \ref{lemma:luna}, so that there are morphisms
\begin{equation}\label{eq:I-ramified}
W_P \to W_P/G_P=V_P \to I^n_{X/C}
\end{equation}
where the left-hand morphism is finite and the right-hand morphism is \'etale.
The pair $(W_P, (W_P)_0)$ is snc and in particular log canonical.
Finite ramified covers preserve this property \cite[Corollary 2.43]{kollar-2013},
so $V_P$ is log canonical. As $V_P \to I^n_{X/C}$ is \'etale,
its image is also log canonical \cite[Proposition 2.15]{kollar-2013}.

Next let $Z\subset I^n_{X/C}$ be a log canonical centre and suppose for
contradiction that $Z$ is contained in the non snc locus $\mathcal{S}\subset
I^n_{X/C}$. Let $D_1,\dots,D_s$ be the components of the special fibre
$(I^n_{X/C})_0$ which contain $Z$. Thus $Z$ is contained in a component
$Y$ of $\cap_{i=1}^s D_i$. Then $Z\subset Y\cap\mathcal{S}$
so
\begin{equation}\label{eq:big-codim}
\codim_{I^n_{X/C}} Z > s
\end{equation}
by Proposition \ref{prop:snc-almost-everywhere}.

Let $P\in \mathcal{H}$ be a point mapping to $Z$ and again
consider the composition \eqref{eq:I-ramified}.
Shrink $W_P$ if necessary so that its image in $I^n_{X/C}$ hits no components of
$(I^n_{X/C})_0$ beyond $D_1,\dots,D_s$.
As $V_P \to I^n_{X/C}$ is \'etale there exists
a log canonical centre $Z'\subset V_P$ mapping dominantly to $Z$
\cite[§2.13]{kollar-2013}.
As $W_P \to V_P$ is finite there is a log canonical centre $Z''\subset
W_P$ mapping dominantly to $Z'$ \cite[§2.42]{kollar-2013}.
For an snc pair such as $(W_P, (W_P)_0)$ (and even for
dlt pairs, see e.g.\ \cite[Theorem 4.16]{kollar-2013}) the log canonical centres are
precisely the irreducible components of the intersections of any number of
components of the boundary divisor. But the divisor $(W_P)_0$ in $W_P$ has precisely
$s$ components, so
\begin{equation}\label{eq:small-codim}
\codim_{W_P} Z'' \leqslant  s.
\end{equation}
Since $Z$, $Z'$ and $Z''$ all have the same dimension
equations \eqref{eq:big-codim} and \eqref{eq:small-codim} contradict each
other.
\end{proof}

\subsubsection{}

We end this section by establishing an identification between the dual complexes
of $\mathcal{H}_0$ and of $(I^n_{X/C})_0$, extending as promised the bijection
between components in Lemma \ref{lem:component-bijection}. Note that Corollary
\ref{cor:codimension} ensures that the dual complex of $(I^n_{X/C})_0$ is well
defined \cite[Definition 8]{dFKX-2017}.

\begin{proposition}\label{prop:bijection}
With notation as in Lemma \ref{lem:component-bijection}, let $\{E_i\}$
be any collection of components of $\mathcal{H}_0$ and let
$\{D_i\}$ be the corresponding collection of components of $(I^n_{X/C})_0$.
Then the rules $Y = \pi(Z)$ and $Z = \pi^{-1}(Y)$ define a bijection between the
components of $\bigcap_i D_i$ and the components of $\bigcap_i E_i$.
\end{proposition}

\begin{proof}
Repeat the argument in Lemma \ref{lem:component-bijection} using that all
components of $\bigcap_i D_i$ have the expected dimension by Corollary
\ref{cor:codimension}, and also noting that each component of $\bigcap_i E_i$ is
$G[n]$-invariant.
\end{proof}

\begin{corollary}\label{cor:same-complex}
There is a natural identification between the dual complexes of $\mathcal{H}_0$
and $(I^n_{X/C})_0$.
\end{corollary}

\subsection{Minimal models}\label{subsec:minimalmodels}

Let $ f \colon X \to C $ be a (projective) strict simple degeneration of relative dimension at most two such that $ \Gamma = \Gamma(X_0) $ allows a bipartite orientation. We moreover assume that $ K_{X/C} $ is relatively semi-ample. Recall that this means, for $ m \gg 0 $, that $(K_{X/C})^{\otimes m} $ is relatively basepoint-free, i.e., the canonical map
$$ f^* f_* ((K_{X/C})^{\otimes m}) \to (K_{X/C})^{\otimes m} $$
is surjective. 

Our first goal in this subsection is to prove that, in this situation, also the relative canonical sheaf of $ I^n_{X/C} $ is relatively semi-ample. This has the interesting consequence that the property of being a minimal model is transfered from $ X \to C $ to the Hilbert scheme degeneration $ I^n_{X/C}$.

\subsubsection{}

To ease notation, we will frequently drop the subscript $ X/C $ and write $ I^n $, $P^n $ and $J^n $ rather than $ I^n_{X/C} $, $P^n_{X/C}$ and $J^n_{X/C}$. A few additional obvious simplifications will be made: 

\begin{itemize}

\item After shrinking $C$ around $0$ (if necessary), we can assume that $C$ is a smooth affine curve, with trivial canonical sheaf. This implies that, for all schemes $ Z $ over $C$ we shall encounter below, we can identify $ K_{Z/C} $ with $ K_Z$. The analogous statement holds for schemes over $C[n]$.

\item It is understood that $m$ is also chosen large enough so that $ (K_{I^n})^{\otimes m} $ is Cartier (recall that $I^n$ is $\mathbb{Q}$-factorial). 

\item Lastly, our assumption on $K_X$ implies that we can find finitely many global sections that generate $(K_X)^{\otimes m} $. Using these sections, we shall produce finitely many global sections that generate $ (K_{I^n})^{\otimes m} $ (which implies that $ (K_{I^n})^{\otimes m} $ is relatively basepoint-free).

\end{itemize}

\subsubsection{}

We first need some preliminary results on the canonical bundle of the expanded degeneration $ X[n] \to C[n] $.

\begin{lemma}\label{lemma:non-vanish0}
The canonical $G[n]$-equivariant projection $g \colon X[n] \to X$ induces an isomorphism $ K_{X[n]} \cong g^* K_{X} $.
\end{lemma}
\begin{proof}
By construction (cf.~e.g. \cite[Section 1]{GHH-2019}), $g$ is a composition of maps $ g_i \colon X[i] \to X[i-1] $, where $g_i$ factors as
$$ X[i] \to X[i-1] \times_{\mathbb{A}^1} \mathbb{A}^2 \to X[i-1] $$
with the first map a small resolution and the second map projection to the first factor. This implies that $ g_i^* K_{X[i-1]} \cong K_{X[i]} $, so our claim follows by induction.
\end{proof}

\begin{lemma}\label{lemma:non-vanish1}
Let $ P_1, \ldots, P_r $ be a finite collection of points in $ X[n]$. Then we can find a $G[n]$-invariant section of $(K_{X[n]})^{\otimes m} $ not vanishing at any $P_i$. 
\end{lemma}
\begin{proof}
By assumption, we can for each $i$ find a section $\tilde{s}_i$ of $(K_{X})^{\otimes m} $ that does not vanish at $ p_i = g(P_i) $. Then $ s_i = g^* \tilde{s}_i $ is a $G[n]$-invariant section of $(K_{X[n]})^{\otimes m}$, which does not vanish at $P_i$ by Lemma \ref{lemma:non-vanish0}. Taking a $\kk$-linear combination $ s = \sum_i \lambda_i s_i $, for sufficiently generic $ \lambda_i \in \kk $, yields a section with the desired properties.
\end{proof}

We consider next the $n$-fold fibred product
$$ X[n]^n = X[n] \times_{C[n]} \ldots \times_{C[n]} X[n]. $$
As an immediate consequence of Corollary \ref{cor:prod}, we note that $ G[n] $ acts freely on the GIT stable locus $ X[n]^{n,ss} $. Thus, the GIT quotient map
$$ \alpha \colon X[n]^{n,ss} \to P^n_{X/C} $$
forms a $G[n]$-torsor. This implies that 
$$ K_{X[n]^{n,ss}} \cong \alpha^* K_{P^n_{X/C}}, $$
and (by descent for torsors) that pullback along $\alpha$ identifies sections of $K_{P^n_{X/C}}$ with $G[n]$-invariant sections of $ K_{X[n]^{n,ss}} $ (and likewise for any power of these sheaves).

\begin{lemma}\label{lemma:non-vanish2}
Let $ Z = (P_1, \ldots, P_n) \in X[n]^{n,ss} $. Then we can find a $G[n]$-invariant and $\mathfrak{S}_n$-invariant section of $(K_{X[n]^{n,ss}})^{\otimes m}$ that does not vanish at $Z$.

\end{lemma}
\begin{proof}
Lemma \ref{lemma:non-vanish1} asserts that there exists a $G[n]$-invariant section $s$ of $(K_{X[n]})^{\otimes m} $ which does not vanish at any point $P_i$. As 
$$ K_{X[n]^n} \cong \mathrm{pr}_1^*( K_{X[n]} ) \otimes \cdots \otimes \mathrm{pr}_n^*( K_{X[n]} ), $$ 
pulling back $s$ from each factor yields an $\mathfrak{S}_n$-invariant section $ s \otimes \ldots \otimes s $. Its restriction to $ X[n]^{n,ss} $ yields a section of $ (K_{X[n]^{n,ss}})^{\otimes m} $ which is both $G[n]$-invariant and $\mathfrak{S}_n$-invariant, and which does not vanish at $Z$. 
\end{proof}

\subsubsection{}

Let $ W $ denote the smooth locus of $ X \to C $. For brevity write $W^{n}$, $W^{(n)} $ and $W^{[n]} $ for the relative product, symmetric product and Hilbert scheme of $ W \to C $. Furthermore write $W^n_*$, $W^{(n)}_*$ and $W^{[n]}_*$ for the respective open loci defined by the condition that the underlying zero cycle has at most one double point, and let $\Delta'$ denote the complement of $W^n_*$ in $W^{n}$. By a slight abuse of notation, we also denote by $\Delta'$ the corresponding loci in $W^{(n)} $ and $W^{[n]} $. In each case the complement $\Delta'$ has codimension $2$.

By Proposition \ref{prop:GIT-open}, $W^{[n]}$ is an open subscheme of $I^n_{X/C}$ whose complement is of  codimension $2$. Thus, $W^{[n]}_*$ is open in $I^n_{X/C}$, and its complement $ \mathcal{D} $ has codimension $2$ as well. Consequently, every section of $ K_{ W^{[n]}_*} $ can be extended uniquely to $ W^{[n]} $, and every section of (a suitable power of) $ K_{ W^{[n]}} $ can be extended uniquely to $I^n_{X/C}$.

We shall make use of the following relative version of a construction by Beauville \cite{beauville-1983} (generalizing Beauville's arguments to the smooth quasi-projective morphism $ W \to C $ is straightforward). There is a commutative
diagram (over $C$)

$$
\xymatrix
{ \widetilde{W^n_*}  \ar[r]^p
\ar[d]^{\pi_1}
& W^{[n]}_* \ar[d]^{\pi_2}\\
 W^n_*  \ar[r]^q & W^{(n)}_*
}
$$
where the vertical arrows are blowups along the (relative) diagonal and the horizontal arrows are quotients by the symmetric group
$\mathfrak{S}_n$. The latter acts with stabilizer groups at most $\ZZ/2\ZZ$. Let $E$ be the exceptional divisor for the blowup $\pi_1$. It is also precisely the ramification locus of $p$, so we have both
\begin{align*}
K_{\widetilde{W^n_*}} &\iso p^*(K_{W^{[n]}_*})(E), \quad\text{and}\\
K_{\widetilde{W^n_*}} &\iso \pi_1^*(K_{W^n_*})(E).
\end{align*}
Thus 
$$ p^*K_{W^{[n]}_*} \cong \pi_1^*K_{W^n_*}. $$

\begin{proposition}\label{prop:semi-ample}
Assume that $K_{X}$ is semi-ample over $C$. Then also $K_{I^n}$ is semi-ample over $C$.
\end{proposition}
\begin{proof}
We want to show the following: there is a fixed integer $m>0$ such that for every point $\xi \in I^n_{X/C} $ there exists a section 
$$ \sigma \in H^0(I^n, (K_{I^n})^{\otimes m}) $$
with $\sigma(\xi) \neq 0$. 

Assume first that $\xi \in W^{[n]}_* $. Since $p^*K_{W^{[n]}_*} \cong \pi_1^*K_{W^n_*}$, it is enough, for a fixed $m$, to produce invariant sections of $\pi_1^*(K_{W^n_*})^{\otimes m}$ not vanishing at a given point in $\widetilde{W^n_*}$. For this in turn, it is enough to produce such a section in $(K_{W^n_*})^{\otimes m}$. Any point in $W^n_*$ is the image under $ \alpha $ of some tuple $ Z = (P_1, \ldots, P_n) \in X[n]^{n,ss} $, and any chosen section with the properties in Lemma \ref{lemma:non-vanish2} descends to an invariant section not vanishing at $ \alpha(Z) $.

Now assume that $\xi \in I^n_{X/C} $ is arbitrary. We select a tuple $ Z = (P_1, \ldots, P_n) \in X[n]^{n,ss} $ such that the image of $\alpha(Z)$ under $ P^n_{X/C} \to J^n_{X/C} $ equals the image of $ \xi $ under $ I^n_{X/C} \to J^n_{X/C} $. By Lemma \ref{lemma:non-vanish2}, we can find an $\mathfrak{S}_n$-invariant section $\sigma'$ of $(K_{P^n})^{\otimes m}$ not vanishing at $\alpha(Z)$. By restriction, we obtain an $\mathfrak{S}_n$-invariant section which descends to a section $\sigma \in H^0(W^{[n]}_*, (K_{W^{[n]}_*})^{\otimes m}) $ which can be extended to a section on $W^{[n]}$, and thus to $ I^n_{X/C} $. This section has the property that there exists a neighbourhood $U$ of $\xi$ (induced by the non-vanishing locus of $\sigma'$) such that $\sigma$ has no zeroes on $U \setminus U \cap \mathcal{D}$. But then, since $\mathcal{D}$ has codimension $2$ in $ I^n_{X/C} $, it follows that also $\sigma(\xi) \neq 0$.
\end{proof}

\subsubsection{Good minimal models}
Recall \cite[2.3.2]{NX-2016} that a dlt model $Y\to C$ is a \emph{good minimal}
model if $Y$ is $\QQ$-factorial and the $\QQ$-Cartier divisor
\begin{equation*}
K_Y + (Y_0)_{\mathrm{red}}
\end{equation*}
is semi-ample over $C$.

\begin{corollary}\label{cor:goodminimal}
Assume that $ X \to C $ is a good minimal model. Then also $ I^n_{X/C} \to C $ is a good minimal model. 
\end{corollary}
\begin{proof}
We established in Proposition \ref{prop:Qfactorial} that $ I^n_{X/C} $ is $\mathbb{Q}$-factorial and in Theorem \ref{thm:dltmodel} that $ I^n_{X/C} \to C $ is a dlt model. As moreover $I^n_{X/C,0}$ is a reduced principal divisor, it follows directly from Proposition \ref{prop:semi-ample} that $ K_{I^n_{X/C}} + I^n_{X/C,0}$ is semi-ample over $C$. 
\end{proof}

\begin{remark}
Assume that $ X/C $ has relative dimension $2$. By similar (easier) arguments as above, one shows that if $K_{X/C}$ is trivial (resp.~torsion), then the same property holds for the relative canonical sheaf of $I^n_{X/C}$ over $C$.
\end{remark}


\section{Computing Dual Complexes}\label{sec:dual-computation}

In \eqref{eqn:smcomm} we constructed three families $P^n_{X/C}$, $J^n_{X/C}$
and $I^n_{X/C}$ of GIT quotients over $C$. In this section we relate the dual
complexes of the degenerate fibres of the three families, in order to get a
more intuitive understanding of their combinatorics. We still assume the
original family $f \colon X \to C$ to be a strict simple degeneration of
relative dimension $d \leqslant 2$. We will first prove that dual complexes
exist for the degenerate fibres of these families, then compute them
explicitly. In particular, we show in Theorem \ref{thm:dualcomplex-hilb} that
the dual complex of $I^n_{X/C}$ is the $n$'th symmetric product of the dual
graph of $X/C$, and in Subsection \ref{sec:symgraph} we utilize this result to
compute the dual complex concretely.

We remark that dual complexes of products of degenerations have also, independently, been studied in \cite{BM17}. In the particular case of K3 surfaces, these results in fact suffice in order to describe the essential skeleton (the skeleton of any good minimal dlt model) of the $n$-th Hilbert scheme, without producing an explicit degeneration. Our primary interest, on the other hand, is to compute the dual complex $\mathcal{D}((I^n_{X/C})_0) $ attached to a Hilbert scheme degeneration $I^n_{X/C}$. For this, only the assumption that $ X \to C $ is a projective strict simple degeneration is necessary. Since the property of being a good minimal dlt model is preserved by $I^n_{X/C}$, Theorem \ref{thm:Hilb-skeleton-main} below describes the essential skeleton also for Hilbert schemes of points on more general classes of surfaces.

\subsection{Existence of dual complexes}

In this section we discuss the existence of dual complexes for the degenerate fibres. We will show that the central fibres of all three families in \eqref{eqn:smcomm} satisfy the requirements in Definition \ref{def:2nd-dual}, hence admit dual complexes.

\begin{lemma}
\label{lem:PRdual}
The family $\psi_p \colon P^n_{X/C} \to C$ in \eqref{eqn:smcomm} is an snc model. In particular, the dual complex $\Gamma_p$ of its degenerate fibre $\psi_p^{-1}(0)$ is well-defined.
\end{lemma}

\begin{proof}
By Corollary \ref{cor:prod}, we see that $(\PrXC)^{ss}$ is smooth, on which the $G[n]$-action is free. It follows that the quotient $P^n_{X/C}$ is smooth. To show that $\psi_p^{-1}(0)$ is an snc divisor, it suffices to show that the intersection of any subset of irreducible components of $\psi_p^{-1}(0)$ is smooth of expected dimension.

Without loss of generality we still assume $C=\mathbb{A}^1$. As a $G[n]$-quotient, $P^n_{X/C}$ has a natural stratification
$$ P^n_{X/C} = \bigcup_{I \subseteq [n+1]} Z_I $$
into locally closed subschemes, in which each stratum is given by
\begin{equation}
\label{eqn:open-dim}
Z_I = \varphi_p^{-1}(U_I) / G[n],
\end{equation}
where $\varphi_p$ is defined as in \eqref{eqn:bigcomm} and $U_I$ as in Paragraph \ref{subsubsec:notation}.

It is clear that $Z_I \subseteq \psi_p^{-1}(0)$ if and only if $I \neq \varnothing$. Theorem \ref{thm:limit-scheme-prod} moreover shows that $I \subseteq J$ if and only if $\bar{Z_I} \supseteq Z_J$. Therefore we have
$$\bar{Z_I} = \bigcup_{I \subseteq J} Z_J = \varphi_p^{-1}(\mathbb{A}_I) / G[n],$$
where $\mathbb{A}_I$ is defined as in Paragraph \ref{subsubsec:notation}. Since $\varphi_p$ is a smooth morphism, $\varphi_p^{-1}(\mathbb{A}_I)$ is smooth, hence its quotient by the free $G[n]$-action is also smooth.

It remains to show that each $\bar{Z_I}$ is of expected dimension $dn +1 - \lvert I \rvert$, where $d = 1 \text{ or } 2$ is the relative dimension of the original family. The dimension can be computed on the open subset $Z_I$ using \eqref{eqn:open-dim}. Since $\varphi_p$ is of relative dimension $dn$, we have
$$ \dim \varphi_p^{-1}(U_I) = dn + \dim U_I = dn + (n+1) - \lvert I \rvert. $$
It follows immediately that
$$ \dim Z_I = \dim \varphi_p^{-1}(U_I) - \dim G[n] = dn + 1 - \lvert I \rvert $$
as desired. This finishes the proof.
\end{proof}

\begin{lemma}
\label{lem:exdual}
The dual complex $\Gamma_s$ of the degenerate fibre $\psi_s^{-1}(0)$ of the family $\psi_s \colon J^n_{X/C} \to C$ in \eqref{eqn:smcomm} is well-defined.
\end{lemma}

\begin{proof}
It suffices to verify that the two conditions in Definition \ref{def:2nd-dual} are satisfied by the degenerate fibre $\psi_s^{-1}(0)$. Following the same reasoning as in Lemma \ref{lem:PRdual}, we have a stratification
$$ J^n_{X/C} = \bigcup_{I \subseteq [n+1]} Z'_I $$
such that
\begin{itemize}
\item[($\ast$)] $Z'_I \subseteq \psi_s^{-1}(0)$ iff $I \neq \varnothing$, and $\bar{Z'_I} \supseteq Z'_J$ iff $I \subseteq J$, 
\item[($\ast\ast$)] $\bar{Z'_I} = \varphi_s^{-1}(\mathbb{A}_I)/G[n]$, and  of expected dimension $dn + 1 - \lvert I \rvert$. 
\end{itemize}
We claim that $\bar{Z'_I}$ is normal for each $I \subseteq [n+1]$. Indeed, by Corollary \ref{cor:prod} we have
$$ \varphi_s^{-1}(\mathbb{A}_I) = \varphi_p^{-1}(\mathbb{A}_I) / \mathfrak{S}_n, $$
which is a finite group quotient of a smooth scheme, hence is normal. It follows that $\bar{Z'_I}$ is also normal as a geometric quotient of a normal scheme.

Since $\bar{Z'_I}$ is normal and equi-dimensional, it follows that every connected component of $\bar{Z'_I}$ must be an irreducible component. Indeed, ($R_1$) fails if two irreducible components meet along a divisor, and ($S_2$) fails if they meet along a locus of higher codimension.

Now we verify the conditions required in Definition \ref{def:2nd-dual}. By ($\ast$), ($\ast\ast$) and the above observation, the irreducible components of $\psi_s^{-1}(0)$ are given by the connected components of every $\bar{Z_I}$ with $|I|=1$, which are all normal of dimension $dn$, hence the first condition in Definition \ref{def:2nd-dual} holds. 

To verify the second condition, take $k$ distinct irreducible components of $\psi_s^{-1}(0)$, say, $C_i \subseteq \bar{Z'}_{\{t_i\}}$ for $1 \leqslant i \leqslant k$. If $t_i = t_{i'}$ for some $i \neq i'$, then their intersection is empty; otherwise, their intersection is either empty, or an open and closed subscheme of $\bar{Z'_I}$ with $I = \{ t_1, \ldots, t_k \}$, which is of expected dimension and satisfies the property that every connected component is irreducible. This concludes theproof of the existence of the dual complex for $\psi_s^{-1}(0)$.
\end{proof}

\begin{remark}
By Theorem \ref{thm:dltmodel}, the family $\psi_h \colon I^n_{X/C} \to C$ is a dlt model. As such, its central fibre satisfies the requirements in Definition \ref{def:2nd-dual}, hence its dual complex $\Gamma_h$ is well-defined.
\end{remark}

\subsection{Dual complex for the degenerate fibre of $P^n_{X/C}$}\label{subsec:dualcomplex-product}

From now on we assume that the dual complex $\Gamma$ for the degenerate fibre of the original family $ f: X \to C $ is an ordered graph (i.e., an ordered $\Delta$-complex with $1$-dimensional geometric realization) with a fixed bipartite orientation. Our next goal is to compute the dual complex $\Gamma_p$ for the degenerate fibre of the family
$$ \psi_p \colon P^n_{X/C} \lra C $$
as in \eqref{eqn:smcomm}. We will show that $\Gamma_p$ is the $n$-fold product $\Gamma^n$ equipped with the canonical structure of a $\Delta$-complex.

\subsubsection{}\label{subsubsec:step12}
We recall some generalities of products of ordered $\Delta$-complexes. In general, one can compute their products by subdividing the product of their geometric realizations via the `shuffling operation'; see \cite[\S 3.B, p.278]{Hat-2002}. Here we give a more detailed account for products of ordered graphs, which is the case we need.

For each $1 \leqslant i \leqslant n$, let $\Gamma_i$ be an ordered graph with geometric realization $\lvert \Gamma_i \rvert$. The product $\Gamma_1 \times \Gamma_2 \times \cdots \times \Gamma_n$ can, 
as a $\Delta$-complex, be given as follows:

\textsc{Step 1.} For each $1 \leqslant i \leqslant n$, let $S_i$ be a cell in $\Gamma_i$ (either $0$-cell or $1$-cell). Then $|\Gamma_1| \times |\Gamma_2| \times \cdots \times |\Gamma_n|$ is naturally divided into `cubes' of the form $S_1 \times S_2 \times \cdots \times S_n$. Among the factors, assume that $S_{i_1}, \ldots, S_{i_k}$ are ordered $1$-cells while the others are $0$-cells, then the product is a $k$-cube, which can be viewed as a subset of $\RR^k$ defined by $0 \leqslant x_{i_t} \leqslant 1$ for each $1 \leqslant t \leqslant k$ (with arrows of the $1$-cells pointing towards the increasing direction). For convenience, we call this step a `cube decomposition' of $|\Gamma_1| \times |\Gamma_2| \times \cdots \times |\Gamma_n|$.

\textsc{Step 2.} We further decompose each cube into ordered simplices, labelled by all possible orders on the set of variables $x_{i_1}, \ldots, x_{i_k}$. More precisely, suppose that $j_1, \ldots, j_k$ is a `shuffling' of the indices $i_1, \ldots, i_k$, then an ordered $k$-simplex corresponding to this shuffling is given by
\begin{equation}
\label{eqn:cell}
0 \leqslant x_{j_1} \leqslant \cdots \leqslant x_{j_k} \leqslant 1.
\end{equation}
It is clear that there are $k!$ possible shufflings and the union of the corresponding simplices is the entire cube. Therefore these $k!$ simplices give a subdivision of the $k$-cube. The faces of the maximal dimensional simplices, i.e. the simplices of smaller dimensions, are defined by similar chains of inequalities with some `$\leqslant$' replaced by `$=$'. The number of equalities in the chain is the codimension of the cell; or equivalently, the number of inequalities is $1$ larger than the dimension of the cell. Such decomposition of cubes is compatible with specialization; that is, the decomposition of a face of a cube is the restriction of the decomposition of the cube to the face. Hence the cube decomposition in \textsc{Step 1} is further decomposed into a $\Delta$-complex.

\begin{definition}
\label{def:delta}
By the product of the ordered graphs $\Gamma_i$, $1 \leqslant i \leqslant n$, we mean the unique $\Delta$-complex
$$\Gamma_1 \times \Gamma_2 \times \cdots \times \Gamma_n$$  
constructed in \textsc{Step 1} and \textsc{Step 2} above.
\end{definition}

\subsubsection{}
We are now ready to compute the dual complex $\Gamma_p$ for $\psi_p^{-1}(0)$.

\begin{proposition}
\label{prop:dualp}
Suppose that the dual complex for $f^{-1}(0)$ is an ordered graph $\Gamma$ with a fixed bipartite orientation. Then the dual complex $\Gamma_p$ for $\psi_p^{-1}(0)$ is $\Gamma^n$.
\end{proposition}

\begin{proof}
First of all, we notice that the vertical maps in the commutative diagram
\begin{equation*}
\xymatrix{
X[n] \times_{C[n]} \cdots \times_{C[n]} X[n] \ar[rr] \ar[d] && C[n] \ar[d] \\
X \times_C \cdots \times_C X \ar[rr] && C
}
\end{equation*}
are $G[n]$-invariant, hence we have a commutative diagram
\begin{equation}
\label{eqn:withtau}
\xymatrix{
P^n_{X/C} \ar[rr]^-{\psi_p} \ar[d]_{\tau_p} && C \ar@{=}[d] \\
X \times_C \cdots \times_C X \ar[rr]^-{g_p} && C
}
\end{equation}
where the vertical map $\tau_p$ sends an $n$-tuple to its support in the original family $f \colon X \to C$. It is an isomorphism except on the central fibres, where it is only surjective. 

We prove the statement in two steps, based on the two steps to construct $\Gamma^n$ in Paragraph \ref{subsubsec:step12}. We will first show that the combinatorial structure of $g_p^{-1}(0)$ can be encoded in the cube decomposition of $|\Gamma|^n$ as described in \textsc{Step 1}; then we prove that the further triangulation via shuffling as described in \textsc{Step 2} is precisely realized by the map $\tau_p$.

\textsc{Step 1.} We explain why the combinatorial structure of $g_p^{-1}(0)$ can be encoded in the cube decomposition of $|\Gamma|^n$.

Any $k$-cube in the cube decomposition of $|\Gamma|^n$ has the form 
$$ S=S_1 \times \cdots \times S_n. $$
Among the factors, $k$ of them are edges, representing components of the double locus of $f^{-1}(0)$, while the other $(n-k)$ are vertices, representing components of $f^{-1}(0)$. We assume that
\begin{itemize}
\item the edge factors are $S_{i_1}, \ldots, S_{i_k}$; 
\item among the vertex factors, $h$ of them belong to $V^+$ while the other $(n-k-h)$ belong to $V^-$.
\end{itemize}
Here we have $k \geqslant 0$, $h \geqslant 0$ and $k+h \leqslant n$.

This cube represents the locus in $g_p^{-1}(0)$ consisting of $n$-tuples
$$ Z = (z_1, \ldots, z_n) \in g_p^{-1}(0), $$
such that $z_i$ lies in a unique component of $f^{-1}(0)$ if $S_i$ is a vertex, or in a component of the double locus of $f^{-1}(0)$ if $S_i$ is an edge. 

\textsc{Step 2.} We study the preimage of the locus represented by the cube $S$ and explain why the map $\tau_p$ in \eqref{eqn:withtau} induces a subdivision of $S$ as described in Paragraph \ref{subsubsec:step12}. 

The finer combinatorial structure in $\psi_p^{-1}(0)$ is encoded by an extra parameter $I$ such that
$$ \varnothing \neq I \subseteq \{ 1, \ldots, n+1 \}. $$
For any such $I$, suppose $|I|=l+1$ for some $0 \leqslant l \leqslant n$. Then one has to replace each component of the double locus of $f^{-1}(0)$ by $l$ ordered $\PP^1$-fibrations, and distribute the $n$ points across the components such that the resulting $n$-tuple is stable (recall the description of stability in Paragraph \ref{subsubsec:prod-comp}). More precisely, using the above notation, $I$ can be written as
$$ I = \{ h+1, h+t_1+1, h+t_1+t_2+1, \ldots, h+t_1+\cdots+t_l+1 \} $$
where $t_1, t_2, \ldots, t_l$ are positive integers such that $ t_1 + \cdots + t_l = k $. The configuration of a smoothly supported $n$-tuple 
$$ Z = (z_1, \ldots, z_n) \in \psi_p^{-1}(0) $$
is encoded, in the notation of Paragraph \ref{subsubsec:prod-comp}, by a tuple 
$$ \mathbf{z} \in V(\Gamma_I) \times \cdots \times V(\Gamma_I) $$
and stability of $ \mathbf{z} $ boils down to the following condition: among the $k$ points $z_{i_1}, \ldots, z_{i_k}$ whose images under $\tau_p$ lie on the double loci in $g_p^{-1}(0)$, $t_r$ of them should lie on the $r$-th inserted $\PP^1$-fibration (among those that replace the corresponding double locus) for every $1 \leqslant r \leqslant l$. If $ \mathbf{z} $ is stable, the point $Z$ belongs to the stratum $ (\mathcal{P}_I)^{\circ}_{\textbf{z}} $ over $U_I$. We observe that there are $ k!$ ways to distribute the points $z_{i_1}, \ldots, z_{i_k}$ on the various $\PP^1$-fibrations so that $Z$ is stable, each arising in the following fashion: let $j_1, \ldots, j_k$ be an arbitrary shuffling of the indices $i_1, \ldots, i_k$. Then we require, for each $1 \leq r \leq l$, that the $t_r$ points indexed by $ j_{t_1 + \ldots + t_{r-1} +1 }, \ldots, j_{t_1 + \ldots + t_r} $ are placed on the $r$-th $\PP^1$-fibration.

Now suppose that $j_1, \ldots, j_k$ is a shuffling of $i_1, \ldots, i_k$, such that for any $p$ with $t_1 + \cdots + t_{r-1} < p \leqslant t_1 + \cdots + t_r$, the point $z_{j_p}$ lies in the $r$-th inserted $\PP^1$-fibration. The locus parametrizing such $n$-tuples is represented by an $l$-cell defined by the chain
\begin{align}
0 &\leqslant x_{j_1} = \cdots = x_{j_{t_1}} \nonumber\\
&\leqslant x_{j_{t_1+1}} = \cdots = x_{j_{t_1+t_2}} \nonumber\\
&\leqslant \cdots \label{eqn:simplex} \nonumber\\
&\leqslant x_{j_{t_1+ \cdots + t_{l-1} +1}} = \cdots = x_{j_{t_1+\cdots+t_{l-1}+t_l}} \nonumber\\
&\leqslant 1. \nonumber
\end{align}

Using Theorem \ref{thm:limit-scheme-prod} we conclude that if $I \subset I'$ with $|I|=l+1$ and $|I'|=l'+1$, then the stratum labelled by $I'$ is contained in the closure of the stratum labelled by $I$. This is consistent with the fact that the $l$-cell corresponding to $I$ is a face of the $l'$-cell corresponding to $I'$, because the latter is defined by replacing some equalities from the former by inequalities (one such replacement for each extra element in $I' \backslash I$) without re-shuffling the variables. This verifies that the $\Delta$-complex $\Gamma^n$ we get following the recipe in Paragraph \ref{subsubsec:step12} reflects precisely the combinatorial structure of $\psi_p^{-1}(0)$, which finishes the proof.
\end{proof}

\subsection{Dual complex for the degenerate fibre of $J^n_{X/C}$}\label{subsec:dualcomplex-symprod}

Now we turn to the dual complex $\Gamma_s$ for the degenerate fibre of the second family
$$ \psi_s \colon J^n_{X/C} \longrightarrow C $$
as in \eqref{eqn:smcomm}. Via a comparison to $\Gamma_p$, we will show that $\Gamma_s$ is the symmetric product $\Gamma^n/\mathfrak{S}_n$.

\begin{proposition}\label{prop:duals}
The map $\xi$ in \eqref{eqn:smcomm} induces a map of dual complexes
$$ \xi_\Gamma \colon \Gamma_p \longrightarrow \Gamma_s, $$
which is the quotient map $\Gamma^n \to \Gamma^n/\mathfrak{S}_n$. 
\end{proposition}

\begin{proof}
Notice that the left column in \eqref{eqn:withtau} is equivariant under the action of the symmetric group $\mathfrak{S}_n$, and that both horizontal maps are invariant under the action of $\mathfrak{S}_n$. Hence we can take the $\mathfrak{S}_n$-quotient of the left column and obtain an expansion of \eqref{eqn:withtau} as follows:
\begin{equation}
\label{eqn:longtau}
\xymatrix{
P^n_{X/C} \ar[rr]_{\xi} \ar@/^2pc/[rrrr]^-{\psi_p} \ar[d]_{\tau_p} && J^n_{X/C} \ar[d]_{\tau_s} \ar[rr]_-{\psi_s} && C \ar@{=}[d] \\
X \times_C \cdots \times_C X \ar[rr] \ar@/_2pc/[rrrr]_-{g_p} && \mathrm{Sym}^n(X/C) \ar[rr]^-{g_s} && C
}
\end{equation}
where the quotient map $\xi$ was explained in Remark \ref{rmk:PtoJ}. We have seen in the proof of Proposition \ref{prop:dualp} that the combinatorics of $g_p^{-1}(0)$ is given by a cube decomposition of $\lvert \Gamma \rvert^n$, and $\tau_p$ induces the standard subdivision of the cubes. Our proof is still divided into two steps.

\textsc{Step 1.} We look at the action of $\mathfrak{S}_n$ on the cube decomposition as described in \textsc{Step 1} of the proof of Proposition \ref{prop:dualp}.

The symmetric group $\mathfrak{S}_n$ acts on $X \times_C \cdots \times_C X$, and hence on $g_p^{-1}(0)$, by permuting the factors. Therefore it also acts on $\Gamma^n$ by permuting the factors and induces an $\mathfrak{S}_n$-action on the cube decomposition of $\lvert \Gamma \rvert^n$. More precisely, an element $\sigma \in \mathfrak{S}_n$ sends a cube $S_1 \times \cdots \times S_n$ to $S_{\sigma(1)} \times \cdots \times S_{\sigma(n)}$, compatible with its parametrization as $[0,1]^k$. Hence $\sigma$ also identifies the standard subdivisions of the source and image cubes. In other words, every $\sigma \in \mathfrak{S}_n$ gives an automorphism of $\Gamma^n$ as an ordered $\Delta$-complex. Therefore the quotient $\Gamma^n/\mathfrak{S}_n$ is again an ordered $\Delta$-complex.

\textsc{Step 2.} We need to show that the $\mathfrak{S}_n$-action on the ordered $\Delta$-complex $\Gamma^n$ is compatible with its action on $\psi_p^{-1}(0)$.

Indeed, this can be checked for an arbitrary $\sigma \in \mathfrak{S}_n$ and an arbitrary cube $S_1 \times \cdots \times S_n$. We follow the notation in \textsc{Step 2} of the proof of Proposition \ref{prop:dualp}. A simplex in such a cube is defined by the chain of inequalities
\begin{align*}
0 &\leqslant x_{j_1} = \cdots = x_{j_{t_1}} \\
&\leqslant x_{j_{t_1+1}} = \cdots = x_{j_{t_1+t_2}} \\
&\leqslant \cdots \\
&\leqslant x_{j_{t_1+ \cdots + t_{l-1} +1}} = \cdots = x_{j_{t_1+\cdots+t_{l-1}+t_l}} \\
&\leqslant 1,
\end{align*}
which represents the closure of the locus of $n$-tuples $(z_1, \ldots, z_n)$, such that $z_{j_p}$ lies in the $r$-th inserted component if $t_1 + \cdots + t_{r-1} < p \leqslant t_1 + \cdots + t_r$, for each value of $p$. The image of this locus under the action of $\sigma$ is the closure of $n$-tuples $(z_1, \ldots, z_n)$, such that $z_{\sigma(j_p)}$ lies in the $r$-th inserted component if $t_1 + \cdots + t_{r-1} < p \leqslant t_1 + \cdots + t_r$, for each value of $p$. This locus is in turn represented by the simplex defined by the chain of inequalities
\begin{align*}
0 &\leqslant x_{\sigma(j_1)} = \cdots = x_{\sigma(j_{t_1})} \\
&\leqslant x_{\sigma(j_{t_1+1})} = \cdots = x_{\sigma(j_{t_1+t_2})} \\
&\leqslant \cdots \\
&\leqslant x_{\sigma(j_{t_1+ \cdots + t_{l-1} +1})} = \cdots = x_{\sigma(j_{t_1+\cdots+t_{l-1}+t_l})} \\
&\leqslant 1,
\end{align*}
which is precisely the image of the previous simplex under the action of $\sigma$. This justifies that the $\mathfrak{S}_n$-action on $\psi_p^{-1}(0)$ is compatible with  its action on the dual complex $\Gamma^n$. Therefore the dual complex of the $\mathfrak{S}_n$-quotient of the former is the $\mathfrak{S}_n$-quotient of the latter, which finishes the proof.
\end{proof}

\subsection{Dual complex for the degenerate fibre of $I^n_{X/C}$}\label{subsec:dualcomplex-hilb}

Finally we analyze the relation between the dual complexes for the central fibres of $I^n_{X/C}$ and $J^n_{X/C}$. The tool we need is the following general observation:

\begin{lemma}
\label{lem:same-dual}
Suppose that there is a morphism of schemes of the same pure dimension
$$ \eta \colon Y \lra Z $$
where both $Y$ and $Z$ satisfy the requirements in Definition \ref{def:2nd-dual}. We write the irreducible components of $Z$ as $Z_s$ for $s \in S$ with $S$ being the index set. Then the morphism $\eta$ induces an isomorphism between the dual complexes of $Y$ and $Z$ if the following two conditions are met:
\begin{itemize}
\item[($\dagger$)] For each $s \in S$, $\eta^{-1}(Z_s)$ is an irreducible component $Y_s$ of $Y$;
\item[($\dagger\dagger$)] For each subset $T \subseteq S$ and each irreducible component $Z'$ of $\cap_{s \in T}Z_s$, $\eta^{-1}(Z')$ is an irreducible component $Y'$ of $\cap_{s \in T}Y_s$.
\end{itemize}
\end{lemma}

\begin{proof}
It is clear that such a morphism $\eta$ induces an isomorphism between the two dual complexes, because their dual complexes are built up by attaching cells (see \cite{dFKX-2017}), and $\eta$ leads to the identical way of attaching cells at every stage.
\end{proof}

Equipped with this observation, we are ready to prove the following:

\begin{theorem}\label{thm:dualcomplex-hilb}
The map $\eta$ in \eqref{eqn:smcomm} induces an isomorphism
$$ \eta_\Gamma \colon \Gamma_h \stackrel{\sim}{\longrightarrow} \Gamma_s. $$
In particular, the dual complex $\Gamma_h$ for the fibre $\psi_h^{-1}(0)$ is $\Gamma^n/\mathfrak{S}_n$.
\end{theorem}

\begin{proof}
We will apply Lemma \ref{lem:same-dual} to the morphism
$$\eta \colon \psi_h^{-1}(0) \lra \psi_s^{-1}(0).$$
Following the discussion in Section \ref{sec:strata}, there is a stratification of $\psi_s^{-1}(0)$ (resp.~$\psi_h^{-1}(0)$) into locally closed subschemes $Z_I$'s (resp.~$W_I$'s) indexed by non-empty subsets $I \subseteq [n+1]$, where 
\begin{align*}
Z_I &= \varphi_s^{-1}(U_I)/G[n], \\
\text{resp.}\quad W_I &= \varphi_h^{-1}(U_I)/G[n].
\end{align*}
Then $\bar{Z_I}$ (resp.~$\bar{W_I}$) is the disjoint union of certain irreducible components of $\lvert I \rvert$-fold intersection loci in $\psi_s^{-1}(0)$ (resp.~$\psi_h^{-1}(0)$).

By Proposition \ref{prop:strat-comp}, the irreducible components of $\bar{Z_I}$ (resp.~$\bar{W_I}$) are labelled by pairs $(\mathbf{b}, \mathbf{s})$ which are stable with respect to $I$ in the sense of Definition \ref{def:num-stable}. Let moreover $I=\{i_1, \ldots, i_r\}$. Then we can write
\begin{align*}
(Z_I)_{(\mathbf{b}, \mathbf{s})} &=  \prod_{v \in V} \mathrm{Sym}^{b_v}(Y_v^\ast) \times \prod_{l=1}^{r-1} \big( \big( \prod_{\gamma \in E}\mathrm{Sym}^{s_{r,l}}(\Delta_I^{\gamma,i_l})^\ast \big) /\mathbb{G}_m \big), \\
\text{resp.} \ (W_I)_{(\mathbf{b}, \mathbf{s})} &=  \prod_{v \in V} \mathrm{Hilb}^{b_v}(Y_v^\ast) \times \prod_{l=1}^{r-1} \big( \big( \prod_{\gamma \in E}\mathrm{Hilb}^{s_{r,l}}(\Delta_I^{\gamma,i_l})^\ast \big) /\mathbb{G}_m \big),
\end{align*}
where $Y_v$ is the irreducible component of $f^{-1}(0)$ labelled by the vertex $v \in V$ and $\Delta_I^{\gamma, i_l}$ is a $\mathbb{P}^1$-fibred inserted component. For every fixed $l$, the multiplicative group $\mathbb{G}_m$ acts on the inserted components $\Delta_I^{\gamma, i_l}$ for all $\gamma \in E$ simultaneously. With this notation we have
$$ \eta^{-1}((Z_I)_{(\mathbf{b}, \mathbf{s})}) = (W_I)_{(\mathbf{b}, \mathbf{s})}. $$
Since each $Y_v$ and each $\Delta_I^{\gamma, i_l}$ is a smooth surface (or curve), the Hilbert-Chow morphism on each of them is birational. It follows that
$$ \eta|_{(W_I)_{(\mathbf{b}, \mathbf{s})}}: (W_I)_{(\mathbf{b}, \mathbf{s})} \to (Z_I)_{(\mathbf{b}, \mathbf{s})} $$
is a birational morphism, hence so is the morphism
$$ \eta|_{(\bar{W}_I)_{(\mathbf{b}, \mathbf{s})}}: (\bar{W}_I)_{(\mathbf{b}, \mathbf{s})} \to (\bar{Z}_I)_{(\mathbf{b}, \mathbf{s})}. $$
Therefore we have verified the condition $(\dagger)$ when $\lvert I \rvert=1$ and $(\dagger\dagger)$ when $\lvert I \rvert >1$. It follows by Lemma \ref{lem:same-dual} that $\psi_s^{-1}(0)$ and $\psi_h^{-1}(0)$ have identical dual complexes.
\end{proof}

\subsection{Symmetric products of graphs}\label{sec:symgraph}

We next demonstrate how to concretely do computations with the dual complex of
the Hilbert scheme degeneration $I^n_{X/C}$. In particular we determine when
this $\Delta$-complex is a simplicial complex.

By Theorem \ref{thm:dualcomplex-hilb} the dual complex $\Gamma_h$ of
$I^n_{X/C}$ is the symmetric product $\Gamma^n/\mathfrak{S}_n$ of the dual
graph $\Gamma$ of $X\to C$, equipped with a bipartite orientation. A first
consequence is that the geometric realization of $\Gamma_h$ (as a topological
space) depends only on the geometric realization of $\Gamma$. We harvest 
some further facts on symmetric products of oriented graphs with the
application to $I^n_{X/C}$ in mind.

\begin{example}
If $\Gamma$ is a cycle (without regard to the orientation), then its geometric
realization is homoeomorphic to a circle, and so the geometric realization of
$\Gamma^n/\mathfrak{S}_n$ is homeomorphic to the $n$'th symmetric product of a circle.
Morton \cite{morton-1967} has shown that the latter is a disk bundle over a circle,
more precisely the trivial disk bundle when $n$ is odd and the non orientable
disk bundle when $n$ is even.
\end{example}

\begin{example}
If $\Gamma$ is a tree then $|\Gamma|$ is contractible and $|\Gamma^n|$ is
contractible in an $\mathfrak{S}_n$-invariant way. Thus the symmetric product
$|\Gamma^n/\mathfrak{S}_n|$ is contractible as well.
\end{example}

\begin{example}
More generally, the homotopy type of $|\Gamma^n/\mathfrak{S}_n|$ depends only
on the homotopy type of $|\Gamma|$ by \cite{liao-1954}. Any graph $\Gamma$ has the
homotopy type of a bouquet of circles. The homotopy type of the $n$'th
symmetric product of a bouquet of $m$ cycles has been determined by Ong
\cite{ong-2003}: if $m\le n$ it has the homotopy type of an $m$-torus, and if
$m>n$, it has the homotopy type of the union of the $\binom{m}{n}$
$n$-dimensional subtori in an $m$-torus obtained by fixing $m-n$ out of the
$m$ circle coordinates.
\end{example}

Theorem \ref{thm:limit-scheme} gives a computationally convenient description
of the dual complex of $I^n_{X/C}$ and hence for the symmetric product
$\Gamma^n/\mathfrak{S}_n$ at least when $\Gamma$ is bipartite. In the following
Lemma we give a direct combinatorial argument showing that the same description
works for symmetric products of arbitrary oriented graphs. Recall that for an
oriented graph $\Gamma=(V,E)$ with vertex set $V$ and arrow set $E$ we write
$E^+(v)$ for the set of incoming arrows at the vertex $v$, and $E^-(v)$ for the
set of outgoing edges.

\begin{lemma}\label{lemma:graphsym}
Let $\Gamma = (V, E)$ be an oriented graph. For each integer $n$, the symmetric product
$\Gamma^n/\mathfrak{S}_n$ is isomorphic to the following $\Delta$-complex:
\begin{itemize}
\item[(I)]
A $k$-cell is a pair $(a,r)$, where $a=\{a_v\}$ is a tuple of nonnegative
integers indexed by the vertices $v$ in $\Gamma$ and $r=\{r_{\gamma,i}\}$ is a
tuple of nonnegative integers indexed by the edges $\gamma$ in $\Gamma$ and the
natural numbers $i$ in the range $1\leqslant i \leqslant k$. These are subject to the
conditions
\begin{equation*}
\sum_v a_v + \sum_{\gamma,i} r_{\gamma,i} = n,
\end{equation*}
and
\begin{equation*}
\sum_\gamma r_{\gamma,i} > 0\quad\text{for all $i$.}
\end{equation*}
\item[(II)]
The $k+1$ facets $(b,s) = d_i(a,r)$, for $0\leqslant i \leqslant k$, are as follows:
\begin{align*}
i=0\colon 
&\begin{cases}
b_v = a_v + \sum_{\gamma\in E(v)^-} r_{\gamma,1}\\
s_{\gamma} = (r_{\gamma,2},\dots,r_{\gamma,k})
\end{cases}\\
0<i<k\colon
&\begin{cases}
b = a\\
s_\gamma = (r_{\gamma,1},\dots, r_{\gamma,i}+r_{\gamma,i+1}\dots, r_{\gamma,k})
\end{cases}\\
i=k\colon
&\begin{cases}
b_v = a_v + \sum_{\gamma\in E(v)^+} r_{\gamma,k} \\
s_\gamma = (r_{\gamma,1},\dots,r_{\gamma,k-1}).
\end{cases}
\end{align*}
\end{itemize}
\end{lemma}

\begin{proof}
Given a $k$-cell $(a,r)$, form the cube $S_1\times S_2\times\cdots\times S_n$ in
$\Gamma^n$, where each $S_i$ is either a vertex $v$ or an edge $\gamma$, such
that each vertex $v$ occurs $a_v$ times, each edge $\gamma$ occurs $\sum_j
r_{\gamma,j}$ times and the ordering is arbitrary. Construct a $k$-simplex in
this cube as follows: let $p_0 = (v_{0,1},v_{0,2},\dots,v_{0,n})$ be the
initial vertex in the cube, i.e.
\begin{equation*}
v_{0,i} =
\begin{cases}
v & \text{if $S_i=v$ is a vertex}\\
\src(\gamma) & \text{if $S_i = \gamma$ is an edge}.
\end{cases}
\end{equation*}
Let $p_1 = (v_{1,1},v_{1,2},\dots,v_{1,n})$ be the vertex obtained from $p_0$
by the following procedure: run through all the edges $\gamma$ and arbitrarily
choose $r_{\gamma,1}$ among the indices $i$ satisfying $S_i = \gamma$. For each
chosen index $i$ replace $v_{0,i} = \src(\gamma)$ by $v_{1,i} = \tgt(\gamma)$;
let $v_{1,j} = v_{0,j}$ for the remaining indices $j$. Similarly define $p_2$
by replacing $r_{\gamma_2}$ among the vertices $v_{1,i} = \src(\gamma)$ by
$v_{2,i} = \tgt(\gamma)$ and so on; eventually $p_k$ is the final vertex in the
cube with
\begin{equation*}
v_{k,i} =
\begin{cases}
v & \text{if $S_i=v$ is a vertex}\\
\tgt(\gamma) & \text{if $S_i = \gamma$ is an edge}.
\end{cases}
\end{equation*}
The $k$-simplex $\langle p_0,\dots,p_k\rangle$ in the cube is a $k$-cell in
$\Gamma^n$, and the $k$-cells obtained from $(a,r)$ in this way for all choices
of orderings along the way form an $\mathfrak{S}_n$-orbit and hence a $k$-cell
in $\Gamma^n/\mathfrak{S}_n$. It is then straight\-forward to verify that the
procedure defines a bijection between the sets of $k$-cells compatible with
face maps.
\end{proof}

\begin{proposition}
Let $\Gamma$ be an oriented graph. Then $\Gamma^n/\mathfrak{S}_n$ is a
simplicial complex for all $n$ if and only if $\Gamma$, without the
orientation, is a tree.
\end{proposition}

\begin{proof}
If $\Gamma$ is not a tree, then it contains a cycle. First suppose it contains an oriented cycle.
Let $n$ be the length of that cycle, and define a $1$-simplex $(a,r)$ in $\Gamma^n/\mathfrak{S}_n$ by $a = 0$ and
\begin{equation*}
r_\gamma =
\begin{cases}
1 & \text{if $\gamma$ is in the cycle}\\
0 & \text{otherwise.}
\end{cases}
\end{equation*}
Then the two vertices of $(a,r)$ coincide, so $\Gamma^n/\mathfrak{S}_n$ is not a simplicial complex.

Next suppose $\Gamma$ contains a non directed cycle.
The cycle thus consists of an even number $2N$ of directed paths with alternating orientation,
\begin{equation*}
\lambda_i\colon \overset{v_i}{\bullet} \to \bullet \to \cdots \to \overset{w_i}{\bullet}
\end{equation*}
and
\begin{equation*}
\lambda'_i\colon \overset{w_i}{\bullet} \leftarrow \bullet \leftarrow \cdots \leftarrow \overset{v_{i+1}}{\bullet}
\end{equation*}
where the subscripts are read modulo $N$. Let $n$ be the length of the cycle
minus $N$.
Define two $1$-cells $(a,r)\ne (a',r')$ in $\Gamma^n/\mathfrak{S}_n$ by
\begin{align*}
a_v &=
\begin{cases}
1 & \text{if $v$ is an inner vertex of $\lambda_i$ for some $i$}\\
0 & \text{otherwise}
\end{cases}\\
r_\gamma &=
\begin{cases}
1 & \text{if $\gamma$ is an edge of $\lambda'_i$ for some $i$}\\
0 & \text{otherwise}
\end{cases}
\end{align*}
and $(a',r')$ similarly with $\lambda_i$ and $\lambda'_i$ interchanged. Then
$(a,r)$ and $(a',r')$ have the same pair of vertices, so
$\Gamma^n/\mathfrak{S}_n$ is not a simplicial complex.

Now suppose $\Gamma$ is a tree. We first show that the $k+1$ vertices of any
$k$-simplex of $\Gamma^n/\mathfrak{S}_n$ are pairwise distinct. If not, there is a
$1$-simplex $(a,r)$ in $\Gamma^n/\mathfrak{S}_n$ whose two vertices coincide. Let $\gamma_1$
be such that $r_{\gamma_1} \ne 0$. For the two vertices of $(a,r)$ to coincide,
there must be an arrow $\gamma_2$ with $\src(\gamma_2) = \tgt(\gamma_1)$ and $r_{\gamma_2}\ne 0$.
Then there must be an arrow $\gamma_3$ with $\src(\gamma_3) = \tgt(\gamma_2)$ and
$r_{\gamma_3}\ne 0$ and so on.  Since $\Gamma$ is finite, the process must
produce a directed cycle in $\Gamma$, which is a contradiction.

It remains to show that, if $\Gamma$ is a tree, then any $k$-cell $(a,r)$ in
$\Gamma^n/\mathfrak{S}_n$ can be reconstructed from its set of (pairwise distinct) $k+1$
vertices. Note that a vertex in $\Gamma^n/\mathfrak{S}_n$ is a $0$-cell
$(c,t)$: here $t$ is the empty tuple and hence we shall write simply $c$. By
iterating the face maps in Lemma \ref{lemma:graphsym} one finds that the vertex
$c = (a,r)_i$ opposite to the facet $d_i(a,r)$ is given by
\begin{equation*}
c_v = a_v + \sum_{j\leqslant i} \sum_{\substack{\gamma\\ \src(\gamma)=v}} r_{\gamma,j}
+ \sum_{j>i} \sum_{\substack{\gamma\\ \tgt(\gamma)=v}} r_{\gamma,j}.
\end{equation*}
It is enough to prove that each simplex $(a,r)$ is uniquely determined by one
of its facets $(b,s)=d_i(a,r)$ together with the remaining vertex $c=(a,r)_i$.
So we consider $(b,s)$ and $c$ as known, and the task is to reconstruct
$(a,r)$.

Consider the case $0<i<k$. Then $a=b$, and from $s$ we know all $r_{\gamma,j}$ except for
$r_{\gamma,i}$ and $r_{\gamma,i+1}$, and also the sum
\begin{equation}\label{eq:r-sum1}
r_{\gamma,i}+r_{\gamma,i+1}
\end{equation}
for all edges $\gamma$. Using $c$, we also know,
for each vertex $v$,
\begin{equation}\label{eq:r-sum2}
\sum_{\substack{\gamma\\ \src(\gamma)=v}} r_{\gamma,i}
+ \sum_{\substack{\gamma\\ \tgt(\gamma)=v}} r_{\gamma,i+1}.
\end{equation}
If $v$ is a leaf, there is only one edge with $v$ as vertex so \eqref{eq:r-sum2}
determines either $r_{\gamma,i}$ or $r_{\gamma,i+1}$, and then
\eqref{eq:r-sum1} determines the other.
Fix a root vertex in $\Gamma$, and complete the proof by descending induction on the distance
from a vertex to the root: for each vertex $v$ (except the root itself), exactly one edge $\gamma$ connects
$v$ to a vertex closer to the root. Thus, by induction, all terms except one in the sum \eqref{eq:r-sum2}
are known, enabling us to reconstruct either $r_{\gamma,i}$ or $r_{\gamma,i+1}$, and thus, with
the aid of \eqref{eq:r-sum1}, both.

Consider the case $i=0$.
Then $s$ determines all $r_{\gamma,j}$ except $r_{\gamma,1}$, and then $b_v$ and $c_v$ determines
\begin{equation*}
a_v + \sum_{\substack{\gamma\\ \src(\gamma)=v}} r_{\gamma,1}
\quad\text{and}\quad
a_v + \sum_{\substack{\gamma\\ \tgt(\gamma)=v}} r_{\gamma,1}
\end{equation*}
If $v$ is a leaf, there is only one edge $\gamma$ with $v$ as vertex so the two sums above
determine $a_v$ and $r_{\gamma,1}$. The proof is completed by descending induction on the distance from $v$
to a chosen root.

The case $i=k$ is similar to $i=0$.
\end{proof}

\subsection{The essential skeleton of a Hilbert scheme degeneration}

\subsubsection{}

Let $ C $ denote a smooth connected curve over $\kk$, and let $ 0 \in C $ be a base point. We fix a local parameter $ t $ at $ 0 $, and put $ C^* = C \setminus \{0\} $. Let $ Z^* \to C^* $ be a smooth and projective family of varieties, with $ K_{Z^*}$ semi-ample over $C^*$. We denote by $ (Z^*)^{an} $ the Berkovich analytification of the base change of $Z^*$ along the obvious map $ \mathrm{Spec}~\kk((t)) \to C^* $.

In what follows, we shall recall some terminology and results from \cite{NX-2016}. Given a non-zero regular pluricanonical form $ \omega $ on $Z^* \times_{C^*} \kk((t))$, one can define a weight function $ wt_{\omega} \colon (Z^*)^{an} \to \mathbb{R} \cup \{\infty\} $. The \emph{Kontsevich-Soibelman skeleton} $ \mathrm{Sk}(Z^*,\omega) $ is by definition the minimality locus of the function $ wt_{\omega} $. 

\begin{definition}
The essential skeleton of $ Z^*$ is the union 
$$ \mathrm{Sk}(Z^*) = \bigcup_{\omega} \mathrm{Sk}(Z^*,\omega) $$
in $ (Z^*)^{an} $, where $\omega$ runs over the set of non-zero regular pluricanonical forms. 
\end{definition}

The essential skeleton can also be described in terms of suitable models of $Z^*$ over $C$. Let $ Z \to C $ be a (not necessarily proper) snc model of $Z^*$. Then the dual complex $ \mathcal{D}((Z_0)_{red}) $ of $Z_0$ can be embedded into $ (Z^*)^{an} $ in a natural way; its image is called the \emph{skeleton} $\mathrm{Sk}(Z)$ associated to $Z$. If $Z$ is proper over $C$, the weight function $ wt_{\omega} $ is piecewise affine when restricted to $\mathrm{Sk}(Z)$, and the union of the closed faces where $ wt_{\omega} $ is minimal is precisely equal to $ \mathrm{Sk}(Z^*,\omega) $. This approach can be extended to the case where $Z \to C$ is only dlt, by setting $ \mathrm{Sk}(Z) = \mathrm{Sk}(Z^{snc}) $. Moreover, by \cite[Thm.~3.3.3]{NX-2016} one has $ \mathrm{Sk}(Z^*) = \mathrm{Sk}(Z^{min}) $ for any good minimal dlt model $ Z^{min} \to C $ of $Z^*$. ($Z^*$ admits a good minimal model by our assumption that $ K_{Z^*} $ is semi-ample over $C^*$.)

\subsubsection{}
Let $X \to C $ be a strict simple degeneration of relative dimension at most two. We assume that the dual graph $ \Gamma = \Gamma(X_0) $ admits a bipartite orientation, and that $K_{X^*}$ is semi-ample over $C^*$.
By Theorem \ref{thm:dltmodel}, $ I^n_{X/C} $ is again a dlt model, and, if we set $ \mathcal{D} = \mathcal{D}((I^n_{X/C})_0)$, Theorem \ref{thm:dualcomplex-hilb} asserts that $ \mathcal{D} = \mathrm{Sym}^n(\Gamma)$. If $X/C$ is moreover a good minimal model, the same holds for $ I^n_{X/C} $, by Corollary \ref{cor:goodminimal}. This immediately gives the following result:

\begin{theorem}\label{thm:Hilb-skeleton-main}
Assume that $X^*/C^*$ admits a strict simple degeneration $X/C$, which is also a good minimal model. Then 
$$ \mathrm{Sk}((X^*)^{[n]}) = \mathrm{Sym}^n(\mathrm{Sk}(X^*)). $$
\end{theorem}

\subsubsection{}

In the remainder of this subsection, we would like to explain that, at least in the case where $K_{X^*}$ is trivial, the assumption that the (projective) strict simple degeneration $X/C$ is also a good minimal model can be dropped in Theorem \ref{thm:Hilb-skeleton-main}. For this purpose, we need to introduce some terminology. Let $ Z \to C $ be a dlt model with reduced special fibre $ Z_0 = \sum_{i \in I} E_i $. Let $ \xi_i $ denote the generic point of $E_i$, and let $\omega$ denote a non-zero regular $m$-canonical form on $Z^*$. We denote by $ \mathrm{ord}_{\xi_i}(\omega) $ the order of vanishing of $\omega$ at $\xi_i$, and we define $\mathrm{min}(\omega) = \mathrm{min}_{i \in I} \{\mathrm{ord}_{\xi_i}(\omega) \} $. We shall say that $E_i$ is \emph{$\omega$-minimal} if $\mathrm{ord}_{\xi_i}(\omega) = \mathrm{min}(\omega)$.

\subsubsection{}
We now return to our strict simple degeneration $X/C$. We write $ X_0 = \sum_{i \in I} E_i $, and fix a pluricanonical form $ \omega $ on $X^*$. 

\begin{definition}
We denote by $ \Gamma(\omega) $ the subcomplex of $ \Gamma $ spanned by the vertices $ v_i $ corresponding to components $ E_i $ that are $\omega$-minimal.
\end{definition}

Pulling back $\omega$ from every projection of the fibred product $ (X^*)^n $, and taking the tensor product, yields an $\mathfrak{S}_n$-invariant form $ \Omega $. By similar arguments as in Subsection \ref{subsec:minimalmodels}, $\Omega$ descends to a pluricanonical form on $ (X^*)^{[n]} = \mathrm{Hilb}^n(X^*/C^*) $, which we continue to call $\Omega$. We first prove a lemma that will allow us to do computations on $(X^*)^n $ rather than on $(X^*)^{[n]}$. 
Let $W$ denote the smooth locus of $ X \to C $; to simplify notation, we still write $ W_0 = \sum_{i \in I} E_i $. The components of $(W^n)_0$ will be denoted 
$$ E_A = E_{a_1} \times \ldots \times E_{a_n}, $$ 
where $A = \{ a_1, \ldots, a_n \}$ runs over the multisets of elements in $I$ of cardinality $n$.

\begin{lemma}\label{lemma:omega-minimal}
Let $ E_A $ be an irreducible component of $ (W^n)_0$. Then 
$$ \mathrm{ord}_{E_A}(\Omega) = \sum_{j=1}^n \mathrm{ord}_{E_{a_j}}(\omega). $$
Thus, $E_A$ is $\Omega$-minimal if and only if $E_{a_j}$ is $ \omega$-minimal for each $ a_j \in A $.
\end{lemma}
\begin{proof}
We shall use the standard formula
$$ K_{W^n/C} = \bigotimes_{j = 1}^n \mathrm{pr}_j^* (K_{W/C}), $$
where $ \mathrm{pr}_j \colon W^n \to W $ denotes the $j$-th projection. Let $ \xi_A $ be the generic point of $E_A$. Then $ \mathrm{pr}_j(\xi_A) = \xi_{a_j} $, and taking the stalk at $ \xi_A $ yields 
$$ (K_{W^n/C})_{\xi_A} = \bigotimes_{j=1}^n ((K_{W/C})_{\xi_{a_j}} \otimes_{\mathcal{O}_{W, \xi_{a_j}}} \mathcal{O}_{W^n, \xi_A}). $$
For every $ i \in I $, $(K_{W/C})_{\xi_i}$ is trivial as $\mathcal{O}_{W, \xi_i}$-module, and we fix a generator $ \omega_i $. Consequently, $(K_{W^n/C})_{\xi_A}$ is generated by the element $ \otimes_{j=1}^n \omega_{a_j} $.

As all $E_i$ appear with multiplicity one in the divisor $ W_0$, our chosen uniformizer $t$ at $ 0 \in C $ also yields a uniformizer in the local ring $\mathcal{O}_{W, \xi_i}$. This means that in $(K_{W/C})_{\xi_i}$, $\omega$ can be written (up to a unit) as $ \omega = t^{m_i} \cdot \omega_i $, for some $ m_i \in \mathbb{Z}$. Then $ \mathrm{ord}_{\xi_i}(\omega) = m_i $, and the lemma follows easily from this description. 
\end{proof}

\begin{proposition}\label{proposition:minskeleton}
Let $ \mathcal{D}(\Omega) $ denote the subcomplex of $\mathcal{D}$ spanned by the vertices corresponding to $\Omega$-minimal components of $(I^n_{X/C})_0$. Then $ \mathcal{D}(\Omega) = \mathrm{Sym}^n(\Gamma(\omega)) $.
\end{proposition}
\begin{proof}
Let $ \Delta \subset W^n $ denote the `big' diagonal. Removing $\Delta$ yields an open subset denoted $ W^n_{\circ} = W^n \setminus \Delta $. We can likewise define open subsets $ W^{(n)}_{\circ} $, resp.~$  W^{[n]}_{\circ} $, of the symmetric product, resp.~the Hilbert scheme. Restricting to these loci, the quotient map
$$ q \colon W^n_{\circ} \to W^{(n)}_{\circ} $$
is \'etale, and the Hilbert-Chow morphism
$$ \pi_2 \colon W^{[n]}_{\circ} \to  W^{(n)}_{\circ} $$
is an isomorphism.

In order to compute the $\Omega$-minimal components of $(I^n_{X/C})_0$, it suffices to work on the open subscheme 
$$W^{[n]}_{\circ} \subset I^n_{X/C}, $$
as it contains the generic point of every component of $(I^n_{X/C})_0$. Moreover, as this is an \'etale local computation, we can replace $W^{[n]}_{\circ}$ by its \'etale cover $W^{n}_{\circ}$, so that Lemma \ref{lemma:omega-minimal} applies. This shows that the $\Omega$-minimal vertices of $\mathcal{D}$ arise precisely by making an unordered selection of $n$ vertices in $\Gamma(\omega)$ (allowing repetition). The fact that the span $ \mathcal{D}(\Omega) $ of these vertices inside $\mathcal{D}$ equals $ \mathrm{Sym}^n(\Gamma(\omega)) $ is immediate from our constructions in Subsections \ref{subsec:dualcomplex-product}, \ref{subsec:dualcomplex-symprod} and \ref{subsec:dualcomplex-hilb}.
\end{proof}

\subsubsection{}

To conclude, we apply the above results to a family $ X^*/C^* $ of surfaces, where we assume that $K_{X^*}$ is trivial.

\begin{corollary}\label{cor:KS-skel}
Let $ X^*$ be as above, and assume moreover that it extends over $C$ to a projective strict simple degeneration $X/C$. Then 
$$ \mathrm{Sk}((X^*)^{[n]}) = \mathrm{Sym}^n(\mathrm{Sk}(X^*)). $$
\end{corollary}
\begin{proof}
Since $K_{X^*}$ is trivial, $\mathrm{Sk}(X^*) = \mathrm{Sk}(X^*, \omega)$ for any choice of generator $\omega$ for the canonical bundle (the skeleton does not change by scaling such a form). As $ K_{(X^*)^{[n]}} $ is trivial as well, we likewise find that $ \mathrm{Sk}((X^*)^{[n]}) = \mathrm{Sk}((X^*)^{[n]}, \Omega) $, where $\Omega$ is the form induced by $ \omega $. 

We claim that $ \mathrm{Sk}((X^*)^{[n]}, \Omega) $ equals the span of the $\Omega$-minimal vertices in $ \mathcal{D} = \mathcal{D}(I^n_{X/C,0})$. Indeed, by \cite[Proposition 3.3.2]{NX-2016}, we have $\mathrm{Sk}((X^*)^{[n]}, \Omega) \subset \mathrm{Sk}(I^n_{X/C})$. Furthermore, we can find a projective birational morphism $ h \colon \mathcal{Z} \to I^n_{X/C} $ such that $ (\mathcal{Z}, \mathcal{Z}_0) $ is snc, and such that $h$ restricts to an isomorphism over the snc locus of $I^n_{X/C}$. Then $ \mathrm{Sk}(I^n_{X/C}) \subset \mathrm{Sk}(\mathcal{Z}) $, and the claim follows from \cite[Theorem 4.7.5]{MN-2015}. Thus $\mathrm{Sk}((X^*)^{[n]}, \Omega) = \mathcal{D}(\Omega)$. We likewise find that $\mathrm{Sk}(X^*, \omega) = \Gamma(\omega)$, so the assertion follows from Proposition \ref{proposition:minskeleton}.
\end{proof}


\section{The symplectic form}\label{sec:symplectic}
Assume that $X \to C$ is a projective strict simple degeneration of relative dimension two such that the dual graph has no odd cycles. In Section \ref{sec:dlt}, we proved that $ I^n_{X/C} \to C $ forms a dlt model (Theorem \ref{thm:dltmodel}) which is moreover minimal if $ X \to C $ is minimal (Corollary \ref{cor:goodminimal}). If $ K_{X/C} $ is in addition trivial (e.g. if $X \to C$ is a type II degeneration of K3 surfaces), then also $ K_{I^n_{X/C}} $ is trivial. 

In this section, we consider instead the GIT stack $ \mathcal{I}^n_{X/C} \to C $. In this case, we can improve the above mentioned results. We first make the (easy) observation that $ \mathcal{I}^n_{X/C} \to C $ is semi-stable (as a DM-stack). This implies that $ \mathcal{I}^n_{X/C} \to C $ is log smooth with respect to the natural divisorial log structures in source and target, hence the sheaves of relative log differentials are locally free. In the situation where $ K_{X/C} $ is trivial, we then explain that this has the interesting consequence that $ \mathcal{I}^n_{X/C} \to C $ carries, in a natural way, a symplectic structure. In other words, working on the level of stacks allows us to describe how the symplectic structure on $ \mathrm{Hilb}^n(X_c) $ degenerates as $ c \in C $ tends to $ 0 \in C $.

\begin{remark} 
A natural question is whether one could have worked directly on the GIT quotient, in order to describe the degeneration of the symplectic structure. However, one can show that $ I^n_{X/C} \to C $ fails in general to be log smooth, due to the presence of (non-sym\-plec\-tic) transversal quotient singularities.
\end{remark}

\subsection{Semi-stable DM-stacks}

\subsubsection{}
Let $S$ be a smooth connected curve of finite type over $\kk$. We fix a closed point $ 0 \in S $. Let $ f \colon Z \to S $ be a flat morphism, locally of finite type. We say that $f$ is \emph{semi-stable} if $Z$ is smooth over $\kk$, $Z_0$ is a reduced divisor with normal crossings and the fibres $Z_s$ are otherwise smooth. Note that, as we do not require the irreducible components of $Z_0$ to be smooth, being semi-stable is an \'etale local notion.

\begin{definition} Let $\mathscr{Z}$ be a Deligne-Mumford stack which is flat and locally of finite type over $S$. We say that $\mathscr{Z} \to S$ is semi-stable if there exists an \'etale atlas $ \mathcal{E} \to \mathscr{Z} $ such that the composition $ \mathcal{E} \to S $ is semi-stable.
\end{definition}

\begin{lemma}\label{lemma:semi-stable-stack}
The stack $ \mathcal{I}^n_{X/C} $ is semi-stable over $C$.
\end{lemma}
\begin{proof}
In Lemma \ref{lemma:luna}, we produced for each GIT stable point $ P \in \mathcal{H} $ a slice $ W_P $ which, in particular, is semi-stable over $ C $. Taking the disjoint union of the $W_P$'s, as $P$ varies over the points in $\mathcal{H}$, we obtain a semi-stable \'etale atlas $ \mathcal{E} \to \mathcal{I}^n_{X/C} $. 
\end{proof}

\subsubsection{}
We will make use of some standard results from log geometry; a short summary of the (easy) facts we need is given below. For proofs and further details, the reader will find a detailed treatment in \cite[Chapter 9]{GR-2015}. Our log structures are defined with respect to the \'etale topology.

Let $ f \colon Z \to S $ be a semi-stable morphism. We denote by $S^{+}$ the scheme $S$ equipped with the divisorial log structure induced by $ \{0\} \subset S $. Likewise, $ Z^{+}$ denotes the scheme $Z$ equipped with the divisorial log structure associated to $Z_0 \subset Z$. Then $f$ induces a morphism $ f^{+} \colon Z^{+} \to S^{+} $ of log schemes. To  $ f^{+} $, one can define a sheaf of relative log differentials, denoted $ \Omega^1_{Z^{+}/S^{+}} $. It is locally free, since $ f^{+} $ is \emph{log smooth} (as $f$ is semi-stable).

Let $ g \colon Z' \to Z $ be an \'etale morphism, and equip also $Z'$ with the log structure associated to $ Z'_0 \subset Z' $. Then $g^{+}$ is log \'etale, and the fundamental short exact sequence of log differentials reduces to an isomorphism
\begin{equation}\label{equation-logdiff}
(g^{+})^* \Omega^1_{Z^{+}/S^{+}} \to \Omega^1_{(Z')^{+}/S^{+}}.
\end{equation}
In particular, let $ Z^{sm} \subset Z $ be the smooth locus of the morphism $f$. Then $\Omega^1_{Z^{+}/S^{+}}$ restricts to $\Omega^1_{(Z^{sm})^{+}/S^{+}}$, and it is straightforward to verify that this sheaf coincides with $\Omega^1_{(Z^{sm})/S}$.

Let $ E \to \mathcal{I}^n_{X/C} $ be an \'etale map, with $E$ a scheme. By Lemma \ref{lemma:semi-stable-stack} $ f_E \colon E \to C $ is then semi-stable, thus the sheaf $ \Omega^1_{E^{+}/C^{+}} $ is locally free of rank $2n$, the relative dimension of $E$ over $C$. Now let $ E_i \to \mathcal{I}^n_{X/C} $, $i=1,2$ be two \'etale maps, and $ g \colon E_1 \to E_2 $ a morphism commuting with the respective maps to $ \mathcal{I}^n_{X/C} $. Then $g$ induces a morphism $g^{+} \colon E_1^+ \to E_2^+ $ of log schemes such that $ f_{E_1}^{+} = f_{E_2}^{+} \circ g^{+} $ and, by (\ref{equation-logdiff}), an isomorphism
$$ (g^{+})^* \Omega^1_{E_2^{+}/C^{+}} \to \Omega^1_{E_1^{+}/C^{+}}. $$
One also checks that the cocycle condition \cite{vistoli-1989} holds for triples of such \'etale open sets. In conclusion, the data 
$$ \{E \to \mathcal{I}^n_{X/C}, \Omega^1_{E^{+}/C^{+}} \} $$
forms a locally free sheaf, denoted $ \Omega^1_{(\mathcal{I}^n_{X/C})^{+}/C^{+}} $, on $ (\mathcal{I}^n_{X/C})_{et} $.

\subsection{Symplectic structure of $\mathcal{I}^n_{X/C} \to C$}\label{subsub:logarithmic}

We now make the additional assumption that $ K_{X/C} $ is trivial. Let us also remark that, as $ X \to C $ is semi-stable we can identify $ K_{X/C} $ with the sheaf of relative logarithmic $2$-forms $ \Omega^2_{X^+/C^+} $. Let us fix a nowhere vanishing section $ \omega $ of $ K_{X/C} $; by a slight abuse of notation, we denote also by $\omega $ the restriction of this form to $X^{sm}$. The same argument as in \cite[6]{beauville-1983}, adapted to the relative setting $ X^{sm} \to C $, shows that $\omega $ induces a $2$-form $ \theta $ in $\Omega^2_{\mathrm{Hilb}^n(X^{sm}/C)/C} $ such that $ \theta^n $ is nowhere vanishing.

Now we consider $ \mathrm{Hilb}^n(X^{sm}/C) $ as an open representable substack of $\mathcal{I}^n_{X/C}$. For any \'etale $ E \to \mathcal{I}^n_{X/C} $ with E a scheme, we let
$$ E^{\circ} \to \mathrm{Hilb}^n(X^{sm}/C) $$
denote the restriction to $ \mathrm{Hilb}^n(X^{sm}/C) $. Then $ E^{\circ} $ is open in $E$, with complement $ E \setminus E^{\circ} $ of codimension $2$.

\begin{lemma}
Let $ \theta_{E^{\circ}} $ denote the pullback of $ \theta$ to $ E^{\circ} $. Then $ \theta_{E^{\circ}} $ extends uniquely to a form $ \theta_{E} $ in $ \Omega^2_{E^{+}/C^{+}} $, and $ \theta_{E}^n $ is nowhere vanishing.
\end{lemma}
\begin{proof}
Let $\{E_{\alpha} \to E \}_{\alpha}$ be an \'etale cover over which $ \Omega^2_{E^{+}/C^{+}} $ trivializes. Since each $E_{\alpha}$ is regular, the pullback of the rational section $ \theta_{E^{\circ}} $ extends uniquely to a global section $ \theta_{E_{\alpha}} $ of $ \Omega^2_{E_{\alpha}^{+}/C^{+}} $. It is straightforward to check that the elements $ \theta_{E_{\alpha}} $ glue to a section $ \theta_{E} $ on $E$, and that $ \theta_{E}^n $ is nowhere vanishing.
\end{proof}

\begin{proposition}
The stack $\mathcal{I}^n_{X/C}$ is proper and semi-stable over $C$. It is moreover symplectic, in the sense that it carries a nowhere degenerate logarithmic $2$-form.
\end{proposition}
\begin{proof}
The sections $ \theta_{E} $ produced in the above lemma are easily seen to be compatible when $ E \to \mathcal{I}^n_{X/C}$ runs over the \'etale open subsets of $\mathcal{I}^n_{X/C}$, and thus define a global section of $\Omega^2_{(\mathcal{I}^n_{X/C})^{+}/C^{+}}$.
\end{proof}


\section{Comparison to Nagai's work}\label{sec:nagai}

\subsection{The comparison}

In this section we compare our construction to Nagai's original degeneration  \cite{nagai-2008}. We start with a strict simple degeneration of surfaces
$$ f: X \to C. $$
Nagai's work in   \cite{nagai-2008} is concerned with degree $2$ Hilbert schemes. He constructed a strict simple degeneration of fourfolds
as a resolution of the relative Hilbert scheme $\Hilb^2(X/C)$. 
His work does not use that the dual graph has 
a bipartite orientation, but as we want to compare it with the GIT construction we now assume that $\Gamma(X_0)$ has a bipratite orientation.  

For simplicity, we denote Nagai's family by
$$ H^2_{X/C} \longrightarrow C. $$
On the other hand, we also have the family
$$ I^2_{X/C} \longrightarrow C, $$
constructed via expanded degenerations. Obviously, over any closed point $t \in C \backslash \{0\}$, the fibres of $I^2_{X/C}$ and $H^2_{X/C}$ are both the Hilbert scheme $\Hilb^2(f^{-1}(t))$, hence the total spaces of both families are birational.

The  two families $I^2_{X/C}$ and $H^2_{X/C}$ can be related by explicit birational maps and the situation can be  summarized in the following diagram
\begin{equation}
\label{eqn:big}
\xymatrix{
\mathrm{Bl}_\Delta(X[2] \times_{C[2]} X[2])^{ss}/G[2] \ar[d] \ar[r]^-{\ZZ_2} & \Hilb^2(X[2]/C[2])^{ss}/G[2] \ar[d] & I^2_{X/C} \ar@{=}[l] \\
(X[2] \times_{C[2]} X[2])^{ss}/G[2] \ar[d]_-\beta \ar[r]^-{\ZZ_2} & \mathrm{Sym}^2(X[2]/C[2])^{ss}/G[2] \ar[d] & \\
X \times_C X \ar[r]^-{\ZZ_2} & \mathrm{Sym}^2(X/C) & \\
\mathrm{Bl}_\Delta(X \times_C X) \ar[u] \ar[r]^-{\ZZ_2} & \Hilb^2(X/C) \ar[u] & H^2_{X/C}. \ar[l]
}
\end{equation}
We will explain the diagram in the following discussion.

\subsection{The lower square}

The lower two rows of \eqref{eqn:big} reflect how the family $H^2_{X/C}$ in \cite{nagai-2008} is constructed. The relative Hilbert scheme $\Hilb^2(X/C)$ is constructed as the $\ZZ_2$-quotient of the blowup of $X \times_C X$ along the diagonal $\Delta$. Nagai observed that the total space has $A_1$-singularities. He resolved these singularities by blowing up $\Hilb^2(X/C)$ along some carefully chosen irreducible components of the central fibre. He also proved that the space $H^2_{X/C}$ obtained in this way is a semistable degeneration.

\begin{remark}
We remark that a related construction can be found in \cite{CK-1999}, in the context of the symmetric square of family of smooth curves degenerating to a nodal curve.
\end{remark}

\subsection{The upper square}

The upper two rows of \eqref{eqn:big} reflect how the family $I^2_{X/C}$ in \cite{GHH-2019} is constructed. This construction is based  on the expanded degeneration 
$$ f[2]: X[2] \to C[2]. $$
On this family one can perform a similar construction of the relative Hilbert scheme as follows
\begin{equation*}
\xymatrix{
\mathrm{Bl}_\Delta(X[2] \times_{C[2]} X[2])^{ss} \ar[d] \ar[r]^-{\ZZ_2} & \Hilb^2(X[2]/C[2])^{ss} \ar[d] \\
(X[2] \times_{C[2]} X[2])^{ss} \ar[r]^-{\ZZ_2} & \mathrm{Sym}^2(X[2]/C[2])^{ss}.
}
\end{equation*}
Here we only take semi-stable pairs of points (or closed subschemes of length $2$)
into consideration. 
In particular, they must be supported on the smooth locus of $f[2]$ by \cite[Theorem 2.10]{GHH-2019}. Therefore the product $(X[2] \times_{C[2]} X[2])^{ss}$ and its diagonal are both smooth, as well as the blowup $\mathrm{Bl}_\Delta(X[2] \times_{C[2]} X[2])^{ss}$. It is also easy to see that the  $G[2]$-action commutes with the $\ZZ_2$-action. By taking the $G[2]$-quotient we get the upper square of  diagram \eqref{eqn:big}.

\subsection{The middle square}

The middle square, in particular the map $\beta$, is the key that links the two constructions. The existence of such a map follows easily from the construction of the expanded degeneration. Indeed, the commutative diagram
\begin{equation}
\label{eqn:deg}
\xymatrix{
X[2] \ar[d] \ar[r] & X \ar[d] \\
C[2] \ar[r] & C
}
\end{equation}
defines a morphism
\begin{equation}
\label{eqn:alpha}
\alpha: (X[2] \times_{C[2]} X[2])^{ss} \to X \times_C X.
\end{equation}
Also notice that the diagram \eqref{eqn:deg} respects the $G[2]$-action on the left column. Therefore $\alpha$ descends to a morphism
$$ \beta: (X[2] \times_{C[2]} X[2])^{ss}/G[2] \to X \times_C X. $$
In the next result we will show that $\beta$ is a resolution of singularities, obtained by blowing up (a disjoint union of) Weil divisors in $X \times_C X$.

\begin{proposition}
The morphism
$$ \beta: (X[2] \times_{C[2]} X[2])^{ss}/G[2] \longrightarrow X \times_C X $$
is a small resolution of singularities given by blowups along Weil divisors.
\end{proposition}

\begin{proof}
To prove the statement, we will construct a blowup $\mathrm{Bl}_T(X \times_C X)$ and show that there is an isomorphism
\begin{equation}
\label{eqn:g}
g: (X[2] \times_{C[2]} X[2])^{ss}/G[2] \to \mathrm{Bl}_T(X \times_C X).
\end{equation}
The proof will be given in several steps.

\textsc{Step 1.} We first construct $\mathrm{Bl}_T(X \times_C X)$.

We denote the irreducible components of $X_0$ by $Y_i$ for  $i \in V$, where $V$ is the vertex set of $\Gamma(X_0)$. Then it is clear that the irreducible components of the central fibre of $X \times_C X$ are given by
$$ Y_{ij}=Y_i \times Y_j, $$
for any $i,j \in V$, each of which is a Weil divisor in the total space $X \times_C X$. Since we assume that $\Gamma(X_0)$ has a bipartite orientation, we write $V_0$ for the set of vertices with no incoming arrows. We write
$$ T = \bigcup_{i,j \in V_0} Y_{ij}. $$
Because of the bipartite orientation, $T$ is a disjoint union of irreducible Weil divisors. Following the idea in \cite[Theorem 4.3]{nagai-2008}, we write the blowup of $X \times_C X$ along $T$ by
\begin{equation}
\label{eqn:pi}
\pi: \textrm{Bl}_T(X \times_C X) \to X \times_C X.
\end{equation}

Take any point $(p,q) \in X \times_C X$. If either $p \in X_0^{sm}$ or $q \in X_0^{sm}$, it is clear that in a neighborhood of $(p,q)$, the total space $X \times_C X$ is smooth, and the central fibre is locally an snc divisor. Therefore the blowup $\pi$ is an isomorphism in such a neighborhood. 

Now we assume $p$ and $q$ are both in the singular locus of $X_0$. Then an open neighborhood of $(p,q)$, say $U \subset X \times_C X$, is parameterized by local coordinates
$$ (x_+, y_+, z_+, x_-, y_-, z_-) $$
subject to the relation
$$x_+y_+=t=x_-y_-.$$
Without loss of generality, we may assume that $x_+=0$ and $x_-=0$ cut out the irreducible components of $X_0$ labelled by vertices in $V_0$ in the neighborhoods of $p$ and $q$ respectively. Then the open neighborhood $\pi^{-1}(U) \subset \textrm{Bl}_T(X \times_C X)$ is obtained by blowing up the ideal $(x_+, x_-)$, hence is given by local coordinates
\begin{equation}
\label{eqn:co1}
(x_+, y_+, z_+, x_-, y_-, z_-, [u:v])
\end{equation}
subject to the relation
$$ [u:v] = [x_+:x_-] = [y_-:y_+]. $$
It is clear from the construction that the total space $\mathrm{Bl}_T(X \times_C X)$ is smooth, and $\pi$ is a small resolution.

\textsc{Step 2.} We claim that the morphism $\alpha$ can be lifted to  a morphism $h$ making the following diagram commutative
\begin{equation}
\label{eqn:pro}
\xymatrix{
(X[2] \times_{C[2]} X[2])^{ss} \ar[rd]_-\alpha \ar[r]^-h & \textrm{Bl}_T(X \times_C X) \ar[d]^\pi \\
 & X \times_C X.
}
\end{equation}

The only potential issue for the existence of the lifting is the exceptional locus of $\pi$. For this purpose, we need to write down local coordinates near a point $(p',q') \in (X[2] \times_{C[2]} X[2])^{ss}$ with $\alpha(p',q') \in X_0^\mathrm{sing} \times X_0^\mathrm{sing}$. By \cite[Section 4.2]{wu-2007} or \cite[Proposition 1.7]{GHH-2019}, a relevant open neighborhood in $X[2]$ is parametrized by coordinates
$$ (x,y,z, t_1, t_2, t_3, [u_1:v_1], [u_2:v_2]) $$
with relations
\begin{align*}
[u_1:v_1]&=[t_1:x]=[y:t_2t_3], \\
[u_2:v_2]&=[t_1t_2:x]=[y:t_3].
\end{align*}
We write $r = [v_1:u_1] \in \PP^1$ and $s = [u_2:v_2] \in \PP^1$. Then a point in the same neighborhood of $X[2]$ is given by coordinates
$$ (t_1, t_2, t_3, r, s, z) $$
with the only relation $rs=t_2$. Therefore a point in $(X[2] \times_{C[2]} X[2])^{ss}$ is given by coordinates
$$ (t_1, t_2, t_3, r_+, s_+, z_+, r_-, s_-, z_-) $$
subject to the relations
$$ r_+s_+=t_2=r_-s_-. $$
Using the above coordinates, we define the morphism locally by
\begin{align}
h: \ (X[2] \times_{C[2]} X[2])^{ss} &\to \mathrm{Bl}_T(X \times_C X) \label{eqn:coor}\\
(t_1, t_2, t_3, r_+, s_+, z_+, r_-, s_-, z_-) &\mapsto (x_+, y_+, z_+, x_-, y_-, z_-, [u:v]) \notag
\end{align}
by requiring
\begin{align}
x_\pm &= r_\pm t_1; \notag\\
y_\pm &= s_\pm t_3; \label{eqn:rel}\\
[u:v] &= [r_+:r_-] = [s_-:s_+]. \notag
\end{align}
It is easy to see that the relations are preserved, hence the morphism is well-defined. 

\textsc{Step 3.} We show that $h$ descends to a morphism $g$ making the following diagram commutative
\begin{equation*}
\xymatrix{
(X[2] \times_{C[2]} X[2])^{ss} \ar[r]^-h \ar[dr]_-{G[2]} & \textrm{Bl}_T(X \times_C X) \\
 & ( X[2] \times_{C[2]} X[2] )^{ss} /G[2]. \ar[u]_g &
}
\end{equation*}

For this purpose we need to check that $h$ is $G[2]$-invariant, which only needs to be cheked on the open neighborhood described  above. Take any element $\sigma = (\sigma_1, \sigma_2, \sigma_3) \in G[2]$ (with $\sigma_1\sigma_2\sigma_3=1$). Its action on a point
$$(t_1, t_2, t_3, r_+, s_+, z_+, r_-, s_-, z_-)$$
is given by
$$(\sigma_1 t_1, \sigma_2 t_2, \sigma_3 t_3, \sigma_1^{-1}r_+, \sigma_3^{-1}s_+, z_+, \sigma_1^{-1}r_-, \sigma_3^{-1}s_-, z_-).$$
By the relations \eqref{eqn:rel}, it is easy to observe that
\begin{align*}
&\ \ \ \ h(\sigma_1 t_1, \sigma_2 t_2, \sigma_3 t_3, \sigma_1^{-1}r_+, \sigma_3^{-1}s_+, z_+, \sigma_1^{-1}r_-, \sigma_3^{-1}s_-, z_-) \\
&= h(t_1, t_2, t_3, r_+, s_+, z_+, r_-, s_-, z_-).
\end{align*}
Therefore $h$ is $G[2]$-invariant hence $g$ is well-defined.

We also recall that $(X[2] \times_{C[2]} X[2])^{ss}$ is smooth because  semi-stability requires the pair of two points to have smooth support. Furthermore, since the two points are ordered, the $G[2]$-action is free, hence $(X[2] \times_{C[2]} X[2])^{ss}/G[2]$ is also smooth. Therefore, $g$ is a morphism between smooth varieties.

\textsc{Step 4.} Finally we need to show that $g$ is an isomorphism. By the $G[2]$-equivariance of both $h$ and $\alpha$ in  diagram \eqref{eqn:pro}, we get a commutative diagram
\begin{equation*}
\xymatrix{
(X[2] \times_{C[2]} X[2])^{ss}/G[2] \ar[rd]_-\beta \ar[r]^-g & \textrm{Bl}_T(X \times_C X) \ar[d]^\pi \\
 & X \times_C X.
}
\end{equation*}

It is clear from the construction in Step 1 that the restriction of the morphism $\pi$ is an isomorphism over the open subset
$$ W=(X \times_C X) \backslash (X_0^{sing} \times X_0^{sing}).$$
By the description of the expanded degeneration in \cite[Section 1.3]{GHH-2019}, it is also easy to see that the restriction of $\beta$ over the same open subset $W \subset X \times_C X$ is an isomorphism. Therefore the restriction 
$$ g|_{\beta^{-1}(W)}: \beta^{-1}(W) \to \pi^{-1}(W) $$
is an isomorphism.

To show that the entire $g$ is an isomorphism, we only need to construct a morphism $g^{-1}$ in an open neighborhood of the exceptional locus $\mathrm{Bl}_T(X \times_C X) \backslash \pi^{-1}(W)$. 
Without loss of generality, we only consider the affine chart given by $v=1$ in the neighborhood parameterized by coordinates \eqref{eqn:co1}.  By \eqref{eqn:coor} we obtain
\begin{align*}
&\ \ \ \ h^{-1}(x_+, y_+, z_+, x_-, y_-, z_-, [u:1]) \\
&= \left\{ \left(\frac{x_-}{\lambda}, u\lambda\mu, \frac{y_+}{\mu}, u\lambda, \mu, z_+, \lambda, u\mu, z_- \right) \middle| \lambda, \mu \in \Gm \right\}.
\end{align*}
This shows that $h$ is a trivial $G[2]$-fibration over this affine chart. Therefore $g^{-1}$ is an isomorphism on this affine chart after passing to the $G[2]$-quotient. This concludes the proof that $g$ is an isomorphism.

\textsc{Step 5.} To summarize, we have exhibited the morphism $\beta$ as the composition of $g$ and $\pi$ in \eqref{eqn:g} and \eqref{eqn:pi}. We proved in Step 1 that $\pi$ is a small resolution of singularities by blowups along Weil divisors in $T$. We also constructed $g$ in Steps 2 and 3, and proved that $g$ is an isomorphism in Step 4. This finishes the proof of the proposition.
\end{proof}

From the proof of the above proposition, it is easy to see that $\beta$ is $\ZZ_2$-equivariant. Therefore we have justified the middle square of the diagram \eqref{eqn:big}.

\subsection{Degenerations via symmetric products} 

Recently, Nagai has independently obtained similar results in \cite{nagai-2018}, \cite{nagai-2017}. Indeed, in \cite{nagai-2018}, Nagai studied the symmetric product of a strict simple degeneration. Based on this and using methods from toric geometry, he constructed a relative minimal model $Y^{(n)}$ for the degeneration family of Hilbert schemes. In \cite[Section 4.9]{nagai-2018}, he compared the minimal model $Y^{(2)}$ with the family $H^2_{X/C}$ (which is denoted by $H^{(2)}$ in his papers) and concluded that they differ by a flop. Then, in \cite{nagai-2017}, he compared $Y^{(n)}$ and the family $I^n_{X/C}$ that we constructed using GIT in \cite{GHH-2019} and concluded that they are isomorphic. The fact that $H^2_{X/C}$ and $I^2_{X/C}$ differ by a flop follows from the above two results.

We point out that both approaches have their own advantages. Nagai's approach gives a toric interpretation of a local model of $I^n_{X/C}$ for all $n$. In the meanwhile, our argument in this section reveals the relation between $H^2_{X/C}$ and $I^2_{X/C}$ from a more global perspective.


\bibliography{hilbdeg}{}
\bibliographystyle{alpha}

\end{document}